\documentclass[11pt]{article}
\usepackage{amsmath,amsfonts,amsthm,amscd,amssymb,graphicx}

\usepackage{tikz}

\usepackage{enumitem}

\newtheorem{proposition}{Proposition}[section]
\newtheorem{remark}{Remark}[section]
\newtheorem{lemma}{Lemma}[section]

\numberwithin{equation}{section}
\newtheorem{theorem}{Theorem}[section]
\newtheorem{corollary}{Corollary}[section]

\usepackage{color}

\newcommand{\RR}{\mathbb{R}}

\newcommand{\beqn}{\begin{equation}}
\newcommand{\eeqn}{\end{equation}}
\newcommand{\bear}{\begin{eqnarray}}
\newcommand{\eear}{\end{eqnarray}}
\newcommand{\bean}{\begin{eqnarray*}}
\newcommand{\eean}{\end{eqnarray*}}

\newcommand{\cE}{\mathcal{E}}

\allowdisplaybreaks

\begin{document}
\title{Uniform in time  lower bound for solutions to a quantum Boltzmann equation of bosons }

\author{Toan T. Nguyen\footnotemark[1] \and Minh-Binh Tran\footnotemark[2] 
}

\renewcommand{\thefootnote}{\fnsymbol{footnote}}

\footnotetext[1]{Department of Mathematics, Pennsylvania State University, State College, PA 16802, USA. \\Email: nguyen@math.psu.edu. 
}

\footnotetext[2]{Department of Mathematics, University of Wisconsin-Madison, Madison, WI 53706, USA. \\Email: mtran23@wisc.edu
}

\maketitle
\begin{abstract} 
In this paper, we consider a quantum Boltzmann equation, which describes the growth of the condensate or the interaction between excited atoms and a condensate. The full form of the Bogoliubov dispersion law is considered, which leads to a detailed study of surface integrals inside the collision operator on energy manifolds. We prove that nonnegative radially symmetric  solutions of the quantum Boltzmann equation are bounded from below by a Gaussian distribution, uniformly in time. \end{abstract}




\section{Introduction}
In 1995, the discovery of Bose-Einstein condensation (BEC) in trapped ultracold atomic gases in 1995 \cite{WiemanCornell,Ketterle} has led to an explosion of research on its properties. A kinetic equation for BECs was first derived by Kirkpatrick and Dorfmann \cite{KD1,KD2},  by  a mean field theory and the Green's function method. Following the path of Kirkpatrick and Dorfmann, several authors have tried to derive  kinetic equations to describe the dynamics of BECs \cite{IG,ArkerydNouri:2012:BCI,MR1837939,ArkerydNouri:2015:BCI,KD2,PomeauBrachetMetensRica,GriffinNikuniZaremba:BCG:2009,ReichlGust:2013:TTF,ReichlGust:2012:CII}. 
In the series of papers \cite{QK1,QK2,QK3},  C.W. Gardiner, P. Zoller and coauthors have formulated the Quantum Kinetic Theory, which is both a genuine kinetic theory and a genuine quantum theory, in terms of the Quantum Kinetic Master Equation (QKME) for bosonic atoms.  In the Quantum Kinetic Theory, the significant quantum aspects are restricted to a few modes, the remaining modes being able to be described in the classical way as in the Boltzmann equation. Indeed, the kinetic aspect of the theory arises
from the decorrelation between different momentum bands. The Quantum Kinetic Theory provides a fully quantum mechanical description of the kinetics of a Bose gas, including the regime of a Bose condensation. In particular, the QKME is capable of describing the formation of the Bose condensate. The QKME contains as limiting cases both the Boltzmann-Norheim (Uehling-Ulenbeck)  equation \cite{EscobedoVelazquez:2015:FTB,UehlingUhlenbeck,Nordheim}, the Gross-Pitaevskii equation, and the condensate growth term. The condensate growth term  is indeed the principal term which gives rise to growth of
the condensate, 
by taking atoms out
of the bath of warmer atoms. 

\bigskip


In the following, we briefly recall the kinetic theory of a gas of bosons, more details can be found in \cite{GriffinNikuniZaremba:BCG:2009,JinBinh,ReichlTran,SofferBinh1}. Bosons  of mass $m$ at temperature $T$ can be regarded as quantum-mechanical
wavepackets whose extent is proportional to a thermal de Broglie wavelength $$\lambda_{dB} =
\left(\frac{2 \pi \hbar^2}{m k_B T}\right)^\frac12$$ describing the position uncertainty associated with the thermal momentum
distribution, in which $k_B$ is the Boltzmann constant and $\hbar$ is the Planck constant.  When the gas temperature $T$ is high, the  de Broglie wavelength $\lambda_{dB}$ is very small and  the weakly interacting gas is similar with a system of ``billiard balls''. The dynamics of the density function of the gas $f(t,r,p)$ - the probability of finding a particle at time $t$, position $r$ and momentum $p$ - is described by the Boltzmann-Norheim (Uehling-Ulenbeck) equation
\begin{equation}\label{BN}
\partial_t f(t,r,p) + p \cdot\nabla_r f(t,r,p) = \mathcal{C}_{22}[f](t,r,p), \ \ \ f(0,r,p)= f_0(t,r,p), \ \ \ (t,r,p) \in\mathbb{R}_+\times\mathbb{R}^3\times\mathbb{R}^3,
\end{equation} whose operator sometimes reads 
\begin{equation}\label{QBHightT}
\begin{aligned}
\mathcal{C}_{22}[f](t,r,p_1)\ =& \ \iiint_{\mathbb{R}^{3}\times\mathbb{R}^{3}\times\mathbb{R}^{3}}\delta({p}_1+{p}_2-{p}_3-{p}_4)\delta(\mathcal{E}_{{p}_1}+\mathcal{E}_{{p}_2}-\mathcal{E}_{{p}_3}-\mathcal{E}_{{p}_4})\times\\
\ & \ \times [(1+\vartheta f_1)(1+\vartheta f_2)f_3f_4-f_1f_2(1+\vartheta f_3)(1+ \vartheta f_4)]d p_2d p_3d p_4,
\end{aligned}
\end{equation}
where $\vartheta$ is proportional to $\hbar^3$, $\mathcal{E}_{{p}}$ is the energy of a particle with momentum $p$ and we use the short-hand notation  $f_j= f(t,r,p_j)$. 

The quantum Boltzmann collision operator \eqref{QBHightT} becomes the classical one in the semiclassical limit, as $\vartheta$ tends to $0$. A consequence of this fact is that at  high temperature, the behavior of the Bose gas is, in some sense, quite similar to classical gases.  Note that, different from  classical Boltzmann collision operators, where the collision kernels are functions depending on the types of particles considered, the derived collision kernel for the quantum Boltzmann collision operator for bosons is 1.

When the temperature $T$ becomes lower, $\lambda_{dB}$ becomes smaller. At the BEC transition temperature $T\approx T_{BEC}$, the  de Broglie wavelength becomes comparable to the distance between
bosons. As a consequence, the atomic wavepackets ``overlap'' and the  atoms become indistinguishable. At this
temperature, bosons undergo a quantum-mechanical phase transition and the Bose-Einstein condensate is formed. The gas is said to be at finite  temperature if $T_{BEC}>T>0$K. At this temperature the trapped Bose gas is composed of two distinct components: the high-density {\it{ Bose-Einstein Condensate}} - being localized at the center
of the trapping potential, and the low-density
cloud of thermally {\it excited atoms,} spreading over a much wider region. The system of the coupling between the BEC and the excited atoms consists equations of the wave function $\Psi(t,r)$ of the BEC, which is a function of time and position $(t,r)$ and the  density function $f(t,r,p)$, which a function of time, position, and momentum of the excited atoms $(t,r,p)$. In such a system (cf. \cite{GriffinNikuniZaremba:BCG:2009,SofferBinh1})

\begin{itemize}
\item The wave function $\Psi$ of the BEC is governed by a Gross-Pitaevskii  equation.
\item The density distribution of the gas is governed by a Boltzmann equation, which has two collision operators (see Figure 1):
\begin{itemize}
\item $\mathcal{C}_{12}$ describes collisions of the BEC and the non-condensate (excited) atoms (condensate growth term).
\item $\mathcal{C}_{22}$ describes collisions between non-condensate (excited) atoms.
\end{itemize}
\end{itemize}

  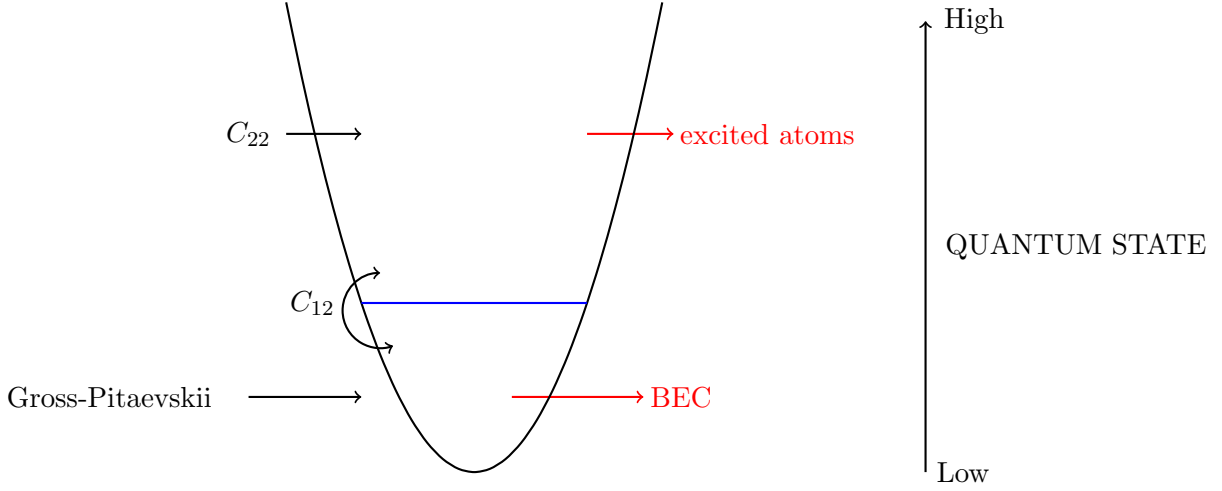
\begin{figure}

\begin{tikzpicture}
      \draw[thick, ->] (-2.5,4.5) -- ++(1,0) node[xshift=-1.5cm] {$C_{22}$};
      \draw[thick, ->, red] (1.5,4.5) -- ++(1.15,0) node[xshift=1.25cm] {excited atoms};
      \draw[thick, ->] (-3,1) -- ++(1.5,0) node[xshift=-3.35cm] {Gross-Pitaevskii };
      \draw[thick, ->, red] (0.5,1) -- ++(1.75,0) node[xshift=.5cm] {BEC};
      \draw[thick, ->] (6,0) node [xshift=.5cm]{Low} -- ++(0,3)  node[xshift=2.0cm] {QUANTUM STATE} -- ++(0,3) node[xshift=.65cm] {High};
      \draw[ thick,domain=-2.5:2.5,smooth,variable=\x,black] plot ({\x},{\x*\x});
      \draw [thick, blue] (-1.5, 2.25) -- (1.5, 2.25);

      \draw[thick, <->] (-1.25,2.65) arc (90:290:.5);
      \node at (-2.15, 2.25) {$C_{12}$};
    \end{tikzpicture}
    \label{fig}\caption{The Bose-Einstein Condensate (BEC) occupies the lowest quantum state and the excitations (thermal cloud) occupy higher quantum states.}
    
   \end{figure}

At very low temperatures, the interaction $\mathcal{C}_{22}$ between bosons themselves is weak as compared to that between excited bosons and the condensate, and thus negligible (cf. \cite{E,EPV,ArkerydNouri:2012:BCI,Arkeryd,ArkerydNouri:AMP:2013,ArkerydNouri:2015:BCI}.)  A simplified version of this model is to consider the coupling system of a spacial homogeneous Boltzmann equation for the density distribution of the excited atoms and an ordinary differential equation describing the density of the BEC superfluid  (cf. \cite{ArkerydNouri:2012:BCI,Arkeryd,ArkerydNouri:AMP:2013,ArkerydNouri:2015:BCI}). In this case, the density distribution function $f(t,p)$ for excited atoms at a time $t\ge 0$ and momentum $p\in \RR^3$ can be described by the following spatially homogenous quantum Boltzmann equation - the condensate growth term
\begin{eqnarray}\label{QuantumBoltzmann}
\frac{\partial f}{\partial t}&=&\mathcal{C}_{12}[f]
\end{eqnarray}
in which $\mathcal{C}_{12}[f]$ denotes the collision integral operator  (see Figure 1) that describes the bosons-condensate interaction \cite{E,IG,Allemand:DOF:2009,Allemand:Thesis:2009,KD2,KD3,ArkerydNouri:2012:BCI,ArkerydNouri:2015:BCI}, defined by 
\begin{equation}\label{E1}
\begin{aligned}
\mathcal{C}_{12}[f] (p) & = 
n_c(t)\iint _{ \RR^3 \times \RR^3}  \Big(\mathcal{R}_{p,p_1,p_2}[f] -\mathcal{R}_{p_1, p, p_2}[f]-\mathcal{R}_{p_2, p, p_1}[f] \Big ) \; dp_1dp_2, \\ 
f(0,p) & = f_0(t,p), \ \ \ (t,p) \in\mathbb{R}_+\times\mathbb{R}^3,
\end{aligned}
\end{equation}
with 
\begin{equation}\label{Kernel} 
\begin{aligned}
\mathcal{R}_{p, p_1, p_2}[f]&= \mathcal K(p, p_1, p_2) \Big( f_1f_2(1+f)-(1+f_1)(1+f_2)f \Big )
\\
\mathcal{K}(p, p_1, p_2)&=K^{12}(p,p_1,p_2)\delta (p-p_1-p_2)\delta (\mathcal{E}(p) -\mathcal{E}(p_1)-\mathcal{E}(p_2)) 
\end{aligned}\end{equation}
using the short-hand notation $f = f(t,p)$ and $f_j= f(t,p_j)$, where $n_c(t)$ is the time dependent density function of the condensate.

Here, $\delta(\cdot)$ denotes the Dirac delta function and $\mathcal{E} (p)$ is the Bogoliubov dispersion law for particle energy, under the assumption that the external potential is zero, which reads
\bear \label{def-E}
\mathcal{E} (p)=|p|\sqrt{\kappa_1  + \kappa_2 |p|^2}, \qquad \kappa_1=\frac{gN_o}{m}>0, \quad \kappa_2=\frac{1}{4m^2}>0,
\eear
for $m$ being the mass of the particles, $g$ the interaction coupling constant,  $N_o$ is assumed to be a constant. In the scope of the paper, for the sake of simplicity, we suppose that
\bear \label{def-EE}
\mathcal{E} (p)=|p|\sqrt{1  + |p|^2}. 
\eear
The transition probability kernel $$K^{12}(p,p_1,p_2)=|A^{12}(|p|,|p_1|,|p_2|)|^2$$ of $\mathcal{C}_{12}$  is given by the scattering amplitude (cf. \cite{E,ReichlGust:2012:CII,ReichlGust:2013:RRA,KD1,KD2,ReichlGust:2013:TTF})
\begin{equation*}
\begin{aligned}
&~~~A^{12}(|p|,|p_1|,|p_2|):=\\
:=&~(u_{p_2}-v_{p_2})(u_{p}u_{p_1}+ v_{p}v_{p_1})+(u_{p_1}-v_{p_1})(u_{p}u_{p_2}+ v_{p}v_{p_2}) - (u_{p}-v_{p})(u_{p_1}v_{p_2}+ v_{p_1}u_{p_2}),
\end{aligned}
\end{equation*}
where 
$$u^2_p=\frac{\frac{p^2}{2m}+gN_o+\mathcal{E}_p}{2\mathcal{E}_p},~~~ u^2_p-v^2_p=1.$$

However, the above form is quite complicated, in several contexts, one usually takes in \eqref{Kernel} the transition probability of the form 
\begin{equation}\label{Kernela}
K^{12}(p,p_1,p_2)={\kappa_0|p||p_1||p_2|}
\end{equation}
for some constant $\kappa_0>0$, which is a valid approximation at sufficiently low temperatures  \cite{EPV,IG,E}. However, {\it in our paper, we will consider a generalized form of \eqref{Kernela}}
 \begin{equation}\label{Kernel}
K^{12}(p,p_1,p_2)={|p|^\rho|p_1|^\rho|p_2|^\rho},\ \ \ \  \forall \rho\in[1,2]
\end{equation}
the constant $\kappa_0$ is omitted for the sake of simplicity.  
 
The Dirac delta function in \eqref{Kernel} ensures the conservation of momentum and energy after collision: 
\begin{equation}\label{cv} p = p_1 + p_2, \qquad \mathcal{E}(p) = \mathcal{E}(p_1) + \mathcal{E}(p_2). \end{equation} 
 In addition,  for the sake of simplicity, we shall take the constants $\kappa_0,\kappa_1,\kappa_2$ all to be one. The results in this paper apply to the general case when the constants are positive. 
 
 \bigskip

The density function $n_c(t)$ of the condensate satisfies the following equation

\begin{equation}\label{BEC}\begin{aligned} 
\dot{n}_c(t)  = -\int_{\mathbb{R}^3} \mathcal{C}_{12}[f](t,p)dp, \ \ \
n_c(0)  = n_{c,0}, \ \ \ t \in\mathbb{R}_+.
\end{aligned} 
\end{equation}
A discussion about the coupling system \eqref{QuantumBoltzmann}-\eqref{BEC} can be found in \cite{ArkerydNouri:2012:BCI}. 

Integrating equation \eqref{BN} with respect to $p$, then add it with equation \eqref{BEC}, we obtain
\begin{equation}
\label{MassConservation}
\partial_t \int_{\mathbb{R}^3}f(t,p)dp + \dot{n}_c(t)  = 0,
\end{equation} 
which implies
 \begin{equation}
\label{MassConservationa}
\int_{\mathbb{R}^3}f(t,p)dp + n_c(t)  = \int_{\mathbb{R}^3}f(0,p)dp + n_c(0) = \mathfrak{M}.
\end{equation}
A consequence of this fact is that both quantities $\int_{\mathbb{R}^3}f(t,p)dp$ and $n_c(t)$ are bounded uniformly in time by the total mass $\mathfrak{M}$ since they are both positive.

The conservation \eqref{MassConservationa} means that for each part of the system - the condensate and the noncondensate - the mass is not conserved. On the other hand, the mass of the full system is conserved.

\bigskip

In dielectric crystals, such as Si and GaAs,  electronic bands  are separated by an energy gap from the conduction band and are completely filled. As a consequence, one could suppress the electronic energy transport and the vibrations of the atoms around their mechanical equilibrium position becomes the dominant contribution to heat transport. In temperatures below the room temperature, these deviations are only a few percent of the lattice constant, and hence weakly anharmonic. Therefore, the obvious theoretical option, proposed by Peierls in 1929  \cite{Peierls:1960:QTS,Peierls:1993:BRK} is to regard the anharmonicities as a small perturbation to the perfectly harmonic crystal, in a certain sense, which leads to a kinetic description of an interacting ``gas of phonons'' using a nonlinear Boltzmann transport equation. This equation, which has exactly the same form with \eqref{QuantumBoltzmann}, is often called the phonon Boltzmann equation, and is normally used to describe the actual computation of the thermal conductivity of dielectric crystals.  The study of \eqref{QuantumBoltzmann} in the context of the quantum theory of solids or anharmonic crystals has also become a subject of growing interests  \cite{Spohn:TPB:2006,AlonsoGambaBinh,CraciunBinh}.

\bigskip
The model \eqref{QuantumBoltzmann} also shares a great similarity with three-wave kinetic models used in the weak turbulence theory \cite{Zakharov:1998:NWA,EscobedoVelazquez:2015:OTT,Spohn:WNW:2010,LukkarinenSpohn:WNS:2011,GambaSmithBinh,germain2017optimal,SofferBinh2,nguyen2017quantum}.  Let us also mention that the derivation of  the wave kinetic equation from the cubic nonlinear Schr\"odinger equation on the torus is also an important topic with rapidly growing interests \cite{germain2015high,FaouGermainHani:TWN:2016,germain2015continuous,germain2016continuous,buckmaster2016effective,buckmaster2016analysis}. 

\bigskip

Above the BEC critical temperature, the homogeneous version of \eqref{BN} takes the form
\begin{equation}\label{UU}
\frac{\partial f}{\partial t}=\mathcal{C}_{22}[f],~~~ f(0,p)=f_0(p), \forall p\in\mathbb{R}^3.
\end{equation}
This equation has a blow-up positive radial solution in the $L^\infty$ norm if the mass of the initial data is highly concentrated around the origin (cf. \cite{EscobedoVelazquez:2015:FTB}). In this temperature regime, the transition probability is the constant $1$ (cf. \cite{ReichlGust:2012:CII,gust2013transport}). The existence of a global weak and measure solution for the equation was studied in \cite{Lu:2004:OID,Lu:2005:TBE,Lu:2013:TBE}. In \cite{BriantEinav:2016:OTC}, local existence and uniqueness results, in the $L^\infty$  norm, were investigated.

Let us mention the  beautiful works \cite{ArkerydNouri:2012:BCI,ArkerydNouri:AMP:2013,ArkerydNouri:2015:BCI}, where the study of $\mathcal{C}_{12}$ has been carried on for the first time. In order to study $\mathcal{C}_{12}$,  the authors of \cite{AlonsoGambaBinh} have developed new methods, based on the techniques of propagation and creation of exponential and polynomial moments for $\mathcal{C}_{12}$. For the case of \eqref{def-E}, the well-posedness theory is obtained recently in \cite{SofferBinh1} for a more general quantum model that contains both $\mathcal{C}_{12}$ and $\mathcal{C}_{22}$. We also mention \cite{Binh9,CraciunBinh} where the convergence to equilibrium is studied for a linearized or discrete version of \eqref{QuantumBoltzmann}. In this paper, assuming the existence of solutions, we prove that {\em positive radial solutions to \eqref{QuantumBoltzmann}-\eqref{def-E} are uniformly bounded below by a Gaussian distribution.}


\bigskip

We emphasize that  in this work the full form of energy functions \eqref{def-E} is considered, which significantly complicates the analysis in treating the collision integral operator $\mathcal{C}_{12}[f]$. The integrals are now reduced to the surface integral on the energy surfaces, dictated by the conservation laws \eqref{cv}, consisting of all points $p_1$ so that 
$$ \cE(p) = \cE(p_1) + \cE(p-p_1)\qquad \mbox{or}\qquad \cE(p+p_1) = \cE(p) + \cE(p_1)$$
for each $p$; see Figures \ref{fig-Sp} and \ref{fig-Sp1} for an illustration of these surfaces. In addition to the complication of dealing with the surface integrals, it is certainly not clear whether the second moment of $f$ on these surfaces is bounded, even if the second moment of $f$ in $\mathbb{R}^3$ is bounded. As a matter of fact, due to this very reason, the simplified energy functions $\mathcal{E}(p) = c |p|$ or $\cE(p) = c|p|^2$ have been used in the literature; see, for instance, \cite{Allemand:DOF:2009,Allemand:Thesis:2009,EscobedoVelazquez:2015:FTB,ArkerydNouri:AMP:2013,EscobedoMischlerValle:HBI:2003} and the references therein. The former energy law leads to line integrals, whereas the latter reduces to integrals on a sphere, as it is the case for the classical Boltzmann equations (e.g., \cite{MR1942465,MouhotVillani:2004:RTS,EPV}).  Up to our knowledge, {\em the current paper is the first time where such a full energy of the form \eqref{def-E} is studied.} We also note that unlike \cite{ArkerydNouri:2012:BCI,Arkeryd,ArkerydNouri:AMP:2013,ArkerydNouri:2015:BCI}, the generalized transition probability $${|p|^\rho|p_1|^\rho|p_2|^\rho},\ \ \ \rho\in[1,2],$$ in this paper is as stated, without being truncated near zero or infinity.

\bigskip

Let us now present the main result of this paper. For $m\ge 1$, introduce the function space $\mathbb{L}_m^1(\RR^3)$, defined by its finite norm 
\begin{equation}\label{L1Space1}
\|f\|_{\mathbb{L}^1_{m}}:= \int_{\mathbb{R}^3}\left(1+\mathcal{E}(p)^m\right) |f(p)|dp,
\end{equation}
with $\cE(p) = |p|\sqrt{1+|p|^2}$. 

\begin{theorem}\label{TheoremLowerBound}
Let $f_0(p)=f_0(|p|)\ge 0$ be a positive radial initial datum in $\mathbb{L}^1_{m} (\RR^3)\cap C(\RR^3)$, for some $m \ge 1$, so that the Cauchy problem \eqref{QuantumBoltzmann}-\eqref{BEC} with the energy \eqref{def-E} has a unique classical positive radial solution $f(t,p) = f(t,|p|)\ge 0$ in $C([0,\infty),\mathbb{L}^1_{m} (\RR^3)\cap C(\RR^3))\cap C^1([0,\infty),\mathbb{L}^1_{m} (\RR^3)\cap C(\RR^3))$, $n_c\in C^1([0,\infty))$ and $n_c$ satisfies
$
n_0\le n_c(t),
$ for some strictly positive constant $n_0$.  Assume that $f_0(p)\geq\theta_0 $ on $B_{R_0}=\{|p|\leq R_0\}$ for some positive constants $\theta_0, R_0$. 
Then, for any time $T>0$, there exist positive constants $\theta_1, \theta_2$ such that 
\begin{equation}\label{TheoremLowerBoundE1}f(t,p)\geq \theta_1\exp(-\theta_2|p|^2), \qquad \forall ~ t\ge T, \quad \forall~p\in \RR^3.\end{equation}
\end{theorem}

\begin{remark} The fact the Cauchy problem \eqref{QuantumBoltzmann}-\eqref{BEC} with the energy \eqref{def-E} has a unique classical positive radial solution $f(t,p) = f(t,|p|)\ge 0$ in $C([0,\infty),\mathbb{L}^1_{m} (\RR^3)\cap C(\RR^3))\cap C^1([0,\infty),\mathbb{L}^1_{m} (\RR^3)\cap C(\RR^3))$, $n_c\in C^1([0,\infty))$ and $n_c$ satisfies
$
n_0\le n_c(t),
$ for some strictly positive constant $n_0$ and for all $t\in[0,\infty)$, has been studied in the revised version of \cite{AlonsoGambaBinh}.
\end{remark}

\begin{remark} The fact that $n_c(t)\ge n_0>0$ for all $t\in[0,\infty)$ means that the density of the BEC, then the BEC itself, are stable and do not diminish as time evolves, if there is no external interference.
\end{remark}

\bigskip

Physically speaking, Theorem \ref{TheoremLowerBound} asserts that the collision operator $\mathcal{C}_{12}[f]$ prevents the excited atoms from falling completely into the condensate. In other words, given a condensate and its thermal cloud, we can prove that there will be some portion of excited atoms which remain outside of the condensate and the density of such atoms will be greater than a Gaussian distribution, uniformly in time $t\geq\tau$, for any time $\tau>0$.

The condition that initial data $f_0(p)$ has positive mass near $\{p=0\}$ is necessary for such a lower bound by a Gaussian to hold, since otherwise $f(t,0)$ would remain zero for all time, as a consequence of $\mathcal{C}_{12}[f](0) =0$, or 
\begin{equation}\label{ZeroMomentum}
\partial_t f(t,0)=0, \qquad \forall ~t\ge 0,
\end{equation}
which implies $$f(t,0)=f_0(0), \qquad \forall ~t\ge 0.$$

\bigskip

In order to prove Theorem \ref{TheoremLowerBound}, it is crucial that we derive bounds on the loss term in the collision operator, which then require bounds on the mass. Moreover, let us emphasize that unlike the classical Boltzmann equation, Equation \eqref{QuantumBoltzmann} does not conserve  the mass. However, the full system \eqref{QuantumBoltzmann}-\eqref{BEC} conserves the mass.

\bigskip

The attempts to quantify the strict positivity of the solution to the Boltzmann equation
are as classical as the mathematical theory of the Boltzmann equation and were first done by Carleman in the pioneering paper \cite{Carleman:1933:TEI}. In this paper, he proved that the solution is bounded from  below by $$\theta_1\exp(-\theta_2|p|^{2+\epsilon}), (\epsilon>0),$$ using a ``spreading
property'' of the collision operator. This result was improved by Pulvirenti and Wennberg \cite{PulvirentiWennberg:1997:MLB}, in which they proved, for hard potentials with cutoff in dimension 3, that the spatially homogeneous solutions in with bounded entropy  satisfy a Gaussian lower bound. The proof is  also based on the spreading property of
the collision operator; however, the optimal decay of the lower bound was obtained by some
improvements of the computations. In \cite{Mouhot:QLB:20015}, Mouhot  proved  an explicit lower bound on the solution to the full Boltzmann equation in the torus, under the assumption of some uniform bounds on some hydrodynamic quantities,  for a broad family of collision kernels including in particular long-range interaction models. The study of lower bounds is a very important subject, not only to understand the qualitative behaviour of solutions to the Boltzmann equation, but also to study the convergence to equilibrium using the so-called ``entropy-entropy production''method \cite{Mouhot:QLB:20015,MR1942465}. This method relies on a control from below uniformly in time on the solutions.

\bigskip

The structure of the paper is as follows:
\begin{itemize}
\item In Section \ref{sec-conv}, we establish the conservation of momentum, energy and the H-theorem of \eqref{QuantumBoltzmann}.  
\item In Section \ref{sec-energy}, we provide the technical estimates on the energy surfaces, which are the basic tools of the paper.

\item Section \ref{sec-moments} is devoted to  prove that the second order energy moment is created and propagated uniformly in time. 
\item Finally, Section 5 is devoted to the proof of the main theorem, Theorem \ref{TheoremLowerBound}. 
\end{itemize}

\section{Conservation laws and the H-theorem}\label{sec-conv}
In this section, we present a few basic properties of smooth solutions of \eqref{QuantumBoltzmann}.  

\begin{lemma}\label{Lemma:WeakFormulation}
For any smooth function $f(p)$, there holds 
\begin{eqnarray*}
\int_{\mathbb{R}^3}\mathcal{C}_{12}[f](p)\varphi(p)dp
=\iiint_{\mathbb{R}^3\times\mathbb{R}^3\times\mathbb{R}^3} \mathcal{R}_{p, p_1, p_2}[f] 
\Big( \varphi(p)-\varphi(p_1)-\varphi(p_2) \Big) \; dpdp_1dp_2
\end{eqnarray*}
for any smooth test function $\varphi$. 
\end{lemma}
\begin{proof}
By the definition \eqref{E1} of $\mathcal{C}_{12}[f]$, we have 
$$
\int_{\mathbb{R}^3}\mathcal{C}_{12}[f](p)\varphi(p)dp 
= \iiint_{\mathbb{R}^3\times\mathbb{R}^3\times\mathbb{R}^3}  \Big(\mathcal{R}_{p, p_1, p_2}[f]-\mathcal{R}_{p_1, p, p_2}[f]-\mathcal{R}_{p_2,p, p_1}[f] \Big ) \varphi(p) \; dp dp_1dp_2. 
$$
By switching the variables $p\leftrightarrow p_1$ and $p\leftrightarrow p_2$ in the second and third integral, respectively, the lemma follows.  
\end{proof}

As a consequence, we obtain the following two corollaries. 

\begin{corollary}[Conservation of momentum and energy]\label{Cor-ConservatioMomentum} Smooth solutions $f(t,p)$ of \eqref{QuantumBoltzmann}, with initial data $f(0,p) =f_0(p)$, satisfy 
\begin{eqnarray}\label{Coro:ConservatioMass}
\int_{\mathbb{R}^3}f(t,p)pdp&=&\int_{\mathbb{R}^3}f_0(p)pdp\\\label{Coro:ConservatioMomentum}
\int_{\mathbb{R}^3}f(t,p)\mathcal{E}(p)dp&=&\int_{\mathbb{R}^3}f_0(p)\mathcal{E}(p)dp 
\end{eqnarray}
for all $t\ge 0$. 
\end{corollary}
\begin{proof} This follows from Lemma \ref{Lemma:WeakFormulation} by taking $\varphi(p) = p$ or $\mathcal{E}(p)$.
 \end{proof}

\begin{corollary}[H-Theorem] Smooth solutions $f(t,p)$ of \eqref{QuantumBoltzmann} satisfy 
 $$\frac{d}{dt}\int_{\mathbb{R}^3}\Big[f\log f-(1+f)\log(1+f)\Big]dp\leq 0.$$
In addition, radially symmetric equilibria  of \eqref{QuantumBoltzmann} must have the following form\begin{equation}\label{def-equilibrium}
f(p)=\frac{1}{e^{c\mathcal{E}(p)}-1},\end{equation}
for some positive constant $c$. 
\end{corollary}
\begin{proof}
First notice that
$$\begin{aligned}
\frac{d}{dt}\int_{\mathbb{R}^3}\left[f\log f-(1+f)\log(1+f)\right]dp
&=\int_{\mathbb{R}^3}\partial_t f\log\left(\frac{f}{f+1}\right)dp.\end{aligned}$$
On the other hand, we write 
$$
\begin{aligned}
\int_{\mathbb{R}^3}\mathcal{C}_{12}[f](p)\varphi(p)dp 
&=\iiint_{\mathbb{R}^3\times\mathbb{R}^3\times\mathbb{R}^3}
\mathcal{K}(p, p_1, p_2) 
(1+f)(1+f_1)(1+f_2)
\\&\quad \times \Big( \frac{f_1}{1+f_1}\frac{f_2}{1+f_2}-\frac{f}{1+f} \Big )  [\varphi(p)-\varphi(p_1)-\varphi(p_2)]dpdp_1dp_2.
\end{aligned}$$
Using Lemma \ref{Lemma:WeakFormulation} with $\varphi(p)=\log\left(\frac{f(p)}{f(p)+1}\right)$ and the fact that $(a-b)\log (\frac ab) \ge 0$, with equality if and only if $a = b$, we obtain 
$$\int_{\mathbb{R}^3}\mathcal{C}_{12}[f](p)\log\left(\frac{f(p)}{f(p)+1}\right) dp\leq 0.$$
This yields the claimed inequality in the H-theorem. In the case of equality, we have 
$$\frac{f(p_1)}{f(p_1)+1}\frac{f(p_2)}{f(p_2)+1}-\frac{f(p)}{f(p)+1}=0,$$
or equivalently, setting   
$h(p)=\log \left(\frac{f(p)}{f(p)+1}\right)$, where $h$ is radially symmetric,
\begin{equation}\label{id-h}h(p_1)+h(p_2)=h(p)\end{equation}
for all $(p, p_1, p_2)$ so that $\mathcal{K}(p, p_1, p_2) \not =0$. In particular, by view of the conservation laws \eqref{cv}, the function $h(p)$ satisfies $h(p_1 + p_2) = h(p_1) + h(p_2)$, for all pairs $(p_1, p_2) \in \mathbb{R}^6$ so that 
$$ \mathcal{E} (p_1 + p_2) = \mathcal{E} (p_1) + \mathcal{E} (p_2) .$$

Define $\mathcal{E}^{-1} (a)$ to be the positive number $\xi$ such that
$$\sqrt{\xi^2+\xi^4}=a.$$

We then have  that $$h \circ \mathcal{E}^{-1} (a + b) = h \circ \mathcal{E}^{-1} (a) + h \circ \mathcal{E}^{-1} (b),$$ for all $p_1$ and $p_2$ such that  $|p_1| = \mathcal{E}^{-1} (a)$ and $|p_2|= \mathcal{E}^{-1} (b)$, with the notice that $h$ is radially symmetric. Since $a,b$ may take arbitrary values in $\mathbb{R}$, this yields $h \circ \mathcal{E}^{-1} (a) =  - c a$ for some positive constant $c$ and for all $a\ge 0$, or equivalently $$h(p) =  - c \mathcal{E}(p),$$ for all $p \in \mathbb{R}^3$. This yields \eqref{def-equilibrium} and hence the H-theorem.\end{proof}

\section{Energy surfaces}\label{sec-energy}


In this section, we study the surface integrals that arise in the collision operator, due to the conservation laws \eqref{cv}. Recall the collision kernel 
$$\mathcal{K}(p, p_1, p_2)= |p|^\rho|p_1|^\rho |p_2|^\rho\delta (p-p_1-p_2) \delta (\mathcal{E} (p)-\mathcal{E} (p_1)-\mathcal{E} (p_2)) $$
with $\delta(\cdot)$ being the Dirac delta function. Thus, the volume element $\mathcal{K}(p, p_1, p_2) dp_1 dp_2$ or $\mathcal{K}(p_1, p, p_2) dp_1 dp_2$ in $\RR^6$ is in fact a two-dimensional surface element. Precisely, introduce the functions
 \begin{equation}\label{def-Hp}\begin{aligned}
H_p(w): = \mathcal{E}(w-p) + \mathcal{E}(w) - \mathcal{E}(p), \ \ \ G_p(w): =\mathcal{E}(p+w) - \mathcal{E}(w) - \mathcal{E}(p) ,
\end{aligned}\end{equation}
with $\cE(w) = |w|\sqrt{1+|w|^2}$, and  the corresponding energy surfaces, dictated by the conservation laws \eqref{cv}, 
\begin{equation}\label{def-Sp}\begin{aligned}
S_p: = \Big \{ w\in \mathbb{R}^d~:~H_p(w)=0 \Big\}
, \qquad S'_p : = \Big \{w\in \mathbb{R}^d~:~ G_p(w)= 0 \Big\}.
\end{aligned}\end{equation}
It follows that the collision operators satisfy 
\begin{equation}\label{intro-Surface123} 
\begin{aligned}
\iint _{ \RR^3 \times \RR^3} \mathcal{R}_{p,p_1,p_2}[f] \; dp_1 dp_2&=  \int_{S_p} \mathcal{R}_{p,p-p_2,p_2}[f] \frac{d\sigma(p_2)}{|\nabla H_p(p_2)|}
\\
\iint _{ \RR^3 \times \RR^3} \mathcal{R}_{p_1,p,p_2}[f] \; dp_1 dp_2 &=  \int_{S_p'} \mathcal{R}_{p+p_2,p,p_2}[f] \frac{d\sigma(p_2)}{|\nabla G_p(p_2)|}
.\end{aligned}\end{equation}
The next two lemmas provide estimates on these surface integrals.

\begin{lemma}\label{lem-Sp} Let $S_p$ be defined as in \eqref{def-Sp}. There are positive constants $c_0, C_0$ so that 
\begin{equation}\label{area-Sp} c_0 |p| \min\{ 1, |p|\}\le \int_{S_p} \frac{d\sigma(w)}{|\nabla H_p(w)|} \le C_0 |p| \min\{ 1, |p|\}  ,\end{equation}
and for $\gamma\ge 0$, 
\begin{equation}\label{lem-Sp-e2}  \int_{S_p \cap B(0,\frac12|p|)} |w-p|^\gamma |w|^\gamma \frac{d\sigma(w)}{|\nabla H_p(w)|}  \ge c_0 |p|^{2\gamma +1}\;  \min\{ 1, |p|\},\end{equation}
uniformly in $p\in \mathbb{R}^3$. 
In addition, for any function $F(\cdot)$, we have 
\begin{equation}\label{lem-Sp-e1}
 \int_{S_p} F(|w|) \frac{d\sigma(w)}{|\nabla H_p(w)|}  \le C_0 \int_0^{|p|} \min\{1,u\} F(u)\; du.\end{equation}
\end{lemma}
\begin{proof} Recall that $S_p$ is the surface consisting of $w$ so that $H_p(w) =0$ or $$\mathcal{E}(w-p) + \mathcal{E}(w)= \mathcal{E}(p)$$ 
with $\cE(w) = |w| \sqrt{1+|w|^2}$. It is clear that $S_p$ is symmetric about $\frac p2$. We will prove that the surface $S_p$ is of the form as illustrated in Figure \ref{fig-Sp}. 
First, we note that $\{0,p\}\subset S_p$, and $|w| \le |p|$ and $|w-p|\le |p|$, for all $w\in S_p$, since $\mathcal{E}(w-p)\le \mathcal{E}(p)$, $ \mathcal{E}(w) \le \mathcal{E}(p)$, and $\mathcal{E}(p)$ is a nonnegative increasing function.

\begin{figure}[t]
\centering
\includegraphics[scale=.45]{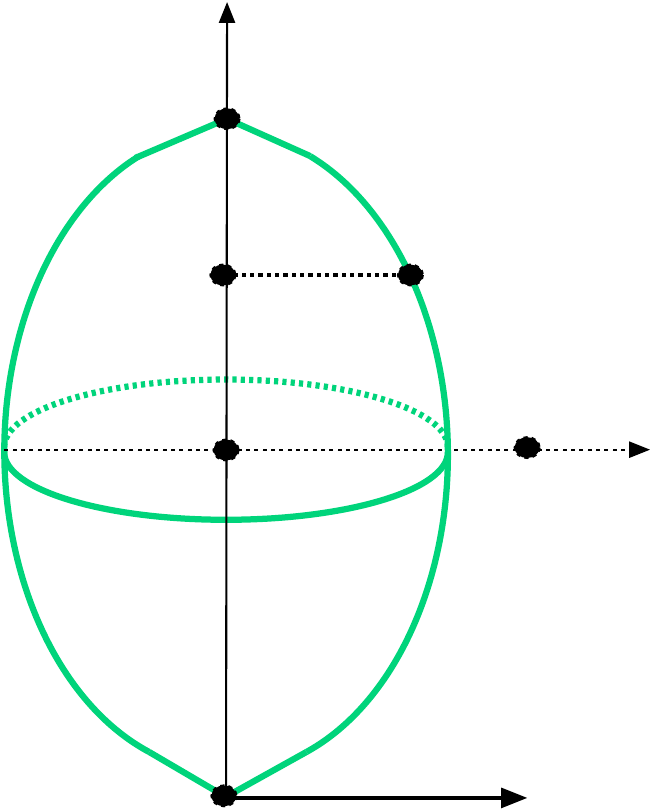}
\put(-95,-10){$0$}
\put(-87,155){$p$}
\put(-25,85){$\frac{|p|}2$}
\put(-30,10){$q$}
\put(-87,85){$\frac{p}2$}
\put(-110,117){$\alpha p$}
\put(-50,120){$w_\alpha$}
\caption{\em Illustrated is the oval surface $S_p$, centered at $\frac p2$ and having $0$ and $p$ as its south and north poles, respectively.}
\label{fig-Sp}
\end{figure}
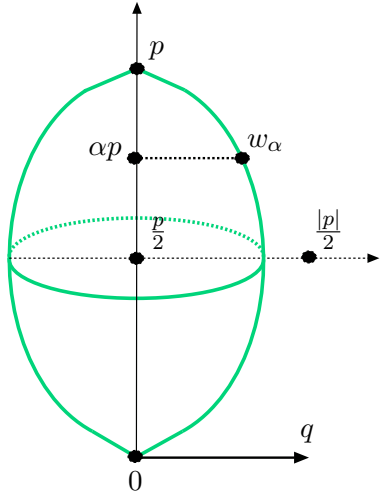

For $w\in S_p$, we write $w=\alpha p+q$, with $p\cdot q =0$. Since $|w| \le |p|$ and $|w-p|\le |p|$, $\alpha \in [0,1]$. In addition, recalling \eqref{def-Hp}, we compute 
\begin{equation}\label{DG} \nabla_w H_p = (1+2|w-p|^2)\frac{w-p}{\cE(w-p)} + (1+2|w|^2)\frac{w}{\cE(w)}.\end{equation}
Thus, $q\cdot \nabla_w H_p >0$. That is, $H_p(w)$ is strictly increasing in any direction that is orthogonal to $p$. This, together with the fact that $H_p(\alpha p) <0$ for $\alpha \in (0,1)$ and $S_p \subset \overline{B(0,|p|)} \cap \overline{B(p,|p|)}$, proves that the surface $S_p$ and the  plane $$\mathcal{P}_\alpha= \Big\{ \alpha p + q, \quad p\cdot q = 0\Big\}$$ 
intersect for each $\alpha \in [0,1]$. In addition, $H_p(\alpha p + q)$ is a radial function in $|q|$, with $q\cdot p=0$. This asserts that the intersection of $S_p$ and $\mathcal{P}_\alpha$ is precisely the circle centered at $\alpha p$ and of a finite radius $|q_\alpha|$, for each $\alpha \in [0,1]$; see Figure \ref{fig-Sp}.

~\\
{\bf Surface parametrization.} Let $p^\perp$ be in $\mathcal{P}_0 = \{ p\cdot q =0\}$ and let $e_\theta$ be the unit vector in $ \mathcal{P}_0$ so that the angle between $p^\perp$ and $e_\theta$ is $\theta$. We parametrize $S_p$ by 
\begin{equation}\label{metric-Sp} S_p = \Big\{ w (\alpha,\theta) = \alpha p + |q_\alpha| e_\theta ~:~ \theta \in [0,2\pi], ~\alpha \in [0,1]\Big\} .\end{equation}
Since $\partial_\theta e_\theta$ is orthogonal to both $p$ and $e_\theta$,  we compute the surface area 
\begin{equation}\label{dS}\begin{aligned}
 d \sigma (w) &= |\partial_\alpha w \times \partial_\theta w | d\alpha d\theta  =  |(p+\partial_\alpha |q_\alpha| e_\theta) \times |q_\alpha| \partial_\theta e_\theta  | d\alpha d\theta
 \\
&= |q_\alpha| |(p+\partial_\alpha |q_\alpha| e_\theta) \times  \partial_\theta e_\theta  | d\alpha d\theta
\\
 &=
 |q_\alpha|\sqrt{|p|^2 +| \partial_\alpha |q_\alpha| |^2}d\alpha d\theta .
 \end{aligned}\end{equation}
To compute $\partial_\alpha |q_\alpha|$, we differentiate the equation $H_p(w_\alpha) = 0$, yielding 
\begin{equation}\label{dG-qa0}
\begin{aligned}
 0 &= \partial_\alpha w_\alpha \cdot \nabla_w H_p(w_\alpha)= |p| e_p  \cdot \nabla_w H_p(w_\alpha) + \partial_\alpha |q_\alpha| e_\theta  \cdot \nabla_w H_p(w_\alpha).
 \end{aligned}\end{equation}
This implies that 
\begin{equation}\label{comp-Dqa} \partial_\alpha |q_\alpha|  = - |p| \frac{e_p  \cdot \nabla_w H_p(w_\alpha)}{ e_\theta  \cdot \nabla_w H_p(w_\alpha)}.\end{equation}
Therefore, we compute 
$$ |p|^2 + |\partial_\alpha |q_\alpha||^2 = |p|^2 \frac{|e_p  \cdot \nabla_w H_p|^2+ |e_\theta  \cdot \nabla_w H_p|^2}{|e_\theta  \cdot \nabla_w H_p|^2} =  |p|^2 \frac{|\nabla_w H_p|^2}{|e_\theta  \cdot \nabla_w H_p|^2},$$
and hence 
\begin{equation}\label{dS-Hp0}
\frac{d\sigma(w)}{|\nabla_wH_p|} = \frac{|p| |q_\alpha| d\alpha d\theta}{|e_\theta \cdot \nabla_w H_p|} .
\end{equation}

~\\
{\bf Surface area.} A direct computation yields 
\begin{equation}\label{comp-DHD}
\begin{aligned}
e_\theta \cdot\nabla_w H_p &= |q_\alpha| \Big[ \frac{1+2|w-p|^2}{\cE(w-p)} +\frac{1+2|w|^2}{\cE(w)}\Big]  .
\end{aligned}\end{equation}
Recalling that $|w|\le |p|$ and $\cE(w) = |w|\sqrt{1+|w|^2}$, and using the fact that $(1+2|p|^2)/\cE(p)$ is decreasing in $|p|$, we compute 
$$\frac{1+2|w-p|^2}{\cE(w-p)} +\frac{1+2|w|^2}{\cE(w)} \ge \frac{1+2|p|^2}{\cE(p)} \ge \min\{1,|p|\}^{-1}. $$
This, \eqref{dS-Hp0}, and \eqref{comp-DHD} prove the upper bound on the surface area \eqref{area-Sp}. As for the lower bound, it suffices to give an estimate for $\alpha\in [0,1/2]$, on which $\alpha |p|\le |w|\le |w-p|$. Thus, in this case, we have 
$$\frac{1+2|w-p|^2}{\cE(w-p)} +\frac{1+2|w|^2}{\cE(w)} \le 2 \frac{1+2|\alpha p|^2}{\cE(\alpha p)} \le C_0 \min\{ 1, \alpha |p|\}^{-1}. $$
The lower bound on the surface area \eqref{area-Sp} follows. 

~\\
{\bf Surface area in $B(0,\frac12|p|)$.} In view of \eqref{comp-Dqa}, \eqref{comp-DHD}, and the identity 
\begin{equation}\label{DHD1}e_p \cdot \nabla_w H_p  =  |p| \Big[ (\alpha-1) \frac{1+2|w-p|^2}{\cE(w-p)} +\alpha \frac{1+2|w|^2}{\cE(w)}\Big] ,\end{equation}
we have $|\partial_\alpha |q_\alpha|| \le |p|^2 |q_\alpha|^{-1}$, which implies
$$|\partial_\alpha |q_\alpha|^2|\le 2|p|^2.$$

Since 
$$|w_\alpha|^2=\alpha^2|p|^2+|q_\alpha|^2,$$
then
$$\partial_\alpha|w_\alpha|^2=2\alpha|p|^2+\partial_\alpha|q_\alpha|^2.$$
Upon recalling that $\alpha \in [0,1]$, $|\partial_\alpha|w_\alpha|^2|\le 4 |p|^2$, and 
$$ |w_\alpha|^2 = \int_0^\alpha \partial_\alpha |w_\alpha|^2 \; d\alpha' \le 4 \alpha |p|^2,$$
which proves that $w_\alpha \in  B(0,\frac12|p|)$ for all $\alpha \in [0,\frac{1}{16}]$. The lower bound \eqref{lem-Sp-e2} follows.

~\\
{\bf Surface integral.} Let us introduce the radial variable $ u = |w_\alpha | = \sqrt{\alpha^2 |p|^2 + |q_\alpha|^2}$. We compute $2u du =\partial_\alpha |w_\alpha|^2 d\alpha$. Hence, \eqref{dS-Hp0} yields   
\begin{equation}\label{dS-Hprad0}
\frac{d\sigma(w)}{|\nabla_wH_p|} = \frac{|p| |q_\alpha|}{|e_\theta \cdot \nabla_w H_p|} \frac{2u du d\theta}{\partial_\alpha |w_\alpha|^2 }.
\end{equation}
In view of \eqref{comp-Dqa}, we compute 
$$ \partial_\alpha |w_\alpha|^2 = 2\alpha |p|^2 + 2 |q_\alpha|\partial_\alpha |q_\alpha| = 2|p|  \frac{\alpha |p| e_\theta  \cdot \nabla_w H_p - |q_\alpha| e_p  \cdot \nabla_w H_p}{ e_\theta  \cdot \nabla_w H_p}$$
in which, using \eqref{comp-DHD} and \eqref{DHD1}, we compute  
$$\alpha |p| e_\theta  \cdot \nabla_w H_p - |q_\alpha| e_p  \cdot \nabla_w H_p =  |p| |q_\alpha| \frac{1+2|w-p|^2}{\cE(w-p)} .$$
Combining, we obtain 
\begin{equation}\label{dS-Hprad}
\frac{d\sigma(w)}{|\nabla_wH_p|} = \frac{\cE(w-p) u du d\theta}{|p|(1+2|w-p|^2)} \le C_0 \min\{1, u\} du d\theta\end{equation}
upon recalling that $|w|\le |p|, |w-p|\le |p|$ for $w\in S_p$ and $\cE(w) = |w|\sqrt{1+|w|^2}$. This proves \eqref{lem-Sp-e1}.  
\end{proof}

\begin{lemma}\label{lem-Sp1} Let $S_p'$ be defined as in \eqref{def-Sp}. 
There are positive constants $c_0,C_0$ so that for any $F(\cdot)$, 
\begin{equation}\label{bd-Sp123} \int_{S_p'} F(|w|) \frac{d\sigma(w)}{|\nabla G_p(w)|}  \le C_0 |p|^{-1}\int_0^\infty  F(u)\; udu,\end{equation}
and 
\begin{equation}\label{low-Sp123}\int_{S_{p}'} F(|w|) \frac{d\sigma(w)}{|\nabla G_p(w)|}  \ge  c_0 \min\{ 1,|p|^{-1}\} \int_0^\infty F(u)\; udu,\end{equation}
for all $p\in \RR^3$. 
\end{lemma}

\begin{proof} Recall that $S_p'$ is the surface that consists of $w$ satisfying $\mathcal{E}(p+w)  =  \mathcal{E}(w) + \mathcal{E}(p)$.  
First, we compute \begin{equation}\label{wp-com} \begin{aligned}
0&=\mathcal{E}(p+w)^2 - \Big( \mathcal{E}(p) + \mathcal{E}(w)\Big)^2  
\\&=    |p+w|^2 +  |p+w |^4  -  (|p|^2 + |w|^2) 
-  (|p|^4 + |w|^4) - 2\mathcal{E}(p) \mathcal{E}(w)
\\&=  2 w\cdot p  +2  w\cdot p (|p|^2 + |w|^2 + |p+w|^2) 
 + 2   |p|^2 |w|^2 - 2\mathcal{E}(p) \mathcal{E}(w).
\end{aligned}\end{equation}
It is clear that $   |p|^2 |w|^2  < \mathcal{E}(p) \mathcal{E}(w)$. This proves that if $w \in S_p' \setminus  \{ 0\}$, then $w\cdot p >0$. 

Next, recall $G_p(w): = \mathcal{E}(p+w) -\mathcal{E}(w) - \mathcal{E}(p) $, with $\cE(p) = |p|\sqrt{1+|p|^2}$. It follows that $G_p(\alpha p) >0$ for $\alpha>0$. In addition, 
we compute   
$$ \nabla_w G_p = \frac{w+p}{|w+p|}\mathcal{E}'(w+p) - \frac{w}{|w|}\cE'(w)$$
and thus the directional derivative of $G_p$ at $w_\alpha = \alpha p+q$, with $p\cdot q =0$, in the direction of $q\not= 0$ satisfies 
$$ q \cdot \nabla_w G_p = |q|^2 \Big[ \frac{\mathcal{E}'(p+w_\alpha)}{|p+w_\alpha|} - \frac{\mathcal{E}'(w_\alpha) }{|w_\alpha|}\Big]  < 0$$
in which we used the fact that $\mathcal{E}'(p)/|p|$ is strictly decreasing in $|p|$. By a view of \eqref{wp-com}, the sign of $G_p(w)$, with $w_\alpha = \alpha p + q$, is the same as that of  
$$
\begin{aligned}
\alpha |p|^2 & \Big( 1 +  (|p|^2 + |w_\alpha|^2 + |p+w_\alpha|^2) \Big) +
    |p|^2 |w_\alpha|^2  - \mathcal{E}(p) \mathcal{E}(w) 
\\ &= \alpha |p|^2  \Big( 1 + 2 (|p|^2 + \alpha |p|^2+ |w_\alpha|^2 ) \Big) 
 -  \frac{ ( |p|^2 +  |p|^4) ( |w_\alpha|^2 +  |w_\alpha|^4) - |p|^4 |w_\alpha|^4}
 {\sqrt{ |p|^2 +  |p|^4}\sqrt{ |w_\alpha|^2 +  |w_\alpha|^4}  +   |p|^2 |w_\alpha|^2 }
 \\ &= \alpha |p|^2  \Big( 1 +  2 (|p|^2 + \alpha |p|^2+ |w_\alpha|^2 ) \Big) 
 -  \frac{  |w_\alpha|^2  |p|^2 +  |w_\alpha|^2|p|^4  +   |p|^2 |w_\alpha|^4 }
 {\sqrt{ |p|^2 +  |p|^4}\sqrt{ |w_\alpha|^2 +  |w_\alpha|^4}  +   |p|^2 |w_\alpha|^2 }.
 \end{aligned}$$
This yields that $G_p(\alpha p + q)<0$ as long as 
$$
\begin{aligned}
 \alpha  & <  \frac{  (1 +  |p|^2)  +   |w_\alpha|^2}
 {\sqrt{ |p|^2 +  |p|^4}\sqrt{\frac{1}{|w_\alpha|^2}+ 1}  +   |p|^2 } \frac{1}{ \Big( 1 +  2 (|p|^2 + \alpha |p|^2+ |w_\alpha|^2 ) \Big)}.
 \end{aligned}$$
Taking $|q| \to \infty$ (and so $|w_\alpha|\to \infty$), we obtain that $\lim_{q\to \infty}G_p(\alpha p + q) <0$ if and only if    
\begin{equation}\label{range-a}\alpha < \alpha_p:= 
\frac12 \frac{ 1}
 { |p|^2 + \sqrt{ |p|^2 +  |p|^4} }.\end{equation}
 In particular, we note that 
\begin{equation}\label{bd-ap}\alpha_p|p| (1+|p|) \le C_0, \qquad \forall ~ p\in \RR^3\end{equation}
for some positive constant $C_0$.
Hence, for positive values of $\alpha$ satisfying \eqref{range-a}, by monotonicity, $G_p(\alpha p) >0$, and the fact that $G_p(\alpha p + q)$ is radial in $|q|$, there is a unique $|q_\alpha|$ so that $G_p(\alpha p + q) =0$, for all $|q| = |q_\alpha|$. For $\alpha> \alpha_p$, $G_p(\alpha p + q)>0$, for all $q$, with $q\cdot p =0$.

\begin{figure}[t]
\centering
\includegraphics[scale=.45]{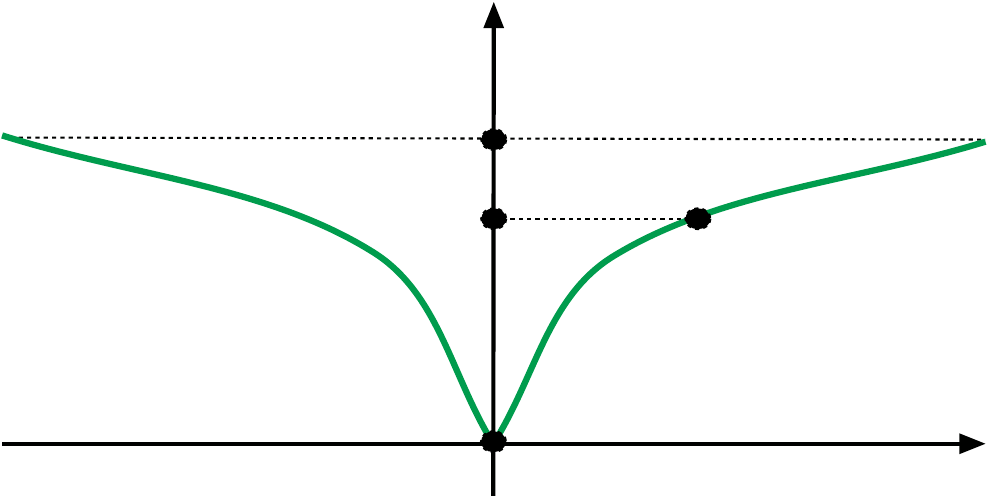}
\put(-105,-2){$0$}
\put(-100,105){$p$}
\put(-8,17){$q$}
\put(-100,82){$\alpha_pp$}
\put(-105,50){$\alpha p$}
\put(-57,55){$w_\alpha$}
\caption{\em Sketched is the trace of $S'_p$ on any two dimensional plane containing $p$.}
\label{fig-Sp1}
\end{figure}
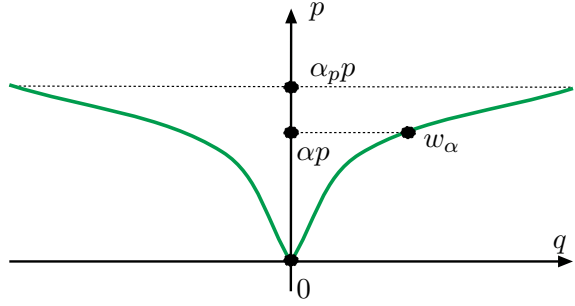

~\\
{\bf Surface parametrization.} To summarize, the surface $S'_p$ can be described as follows (see Figure \ref{fig-Sp1}):
\begin{equation}S'_p =\Big \{ w(\alpha,\theta) = \alpha p + |q_\alpha| e_\theta~:~\alpha \in [0, \alpha_p),~ \theta \in [0,2\pi]\Big\},\end{equation}
in which $\alpha_p$ and $|q_\alpha|$ are defined as above and $e_\theta$ denotes the unit vector rotating around $p$ and on the orthogonal plane to $p$. 

~\\
{\bf Surface integral.} Recalling \eqref{dS-Hprad0}, the surface integral is computed by 
\begin{equation}\label{dS-Hprad01}
\frac{d\sigma(w)}{|\nabla_wG_p|} = \frac{|p| |q_\alpha|}{|e_\theta \cdot \nabla_w G_p|} \frac{2u du d\theta}{\partial_\alpha |w_\alpha|^2 },
\end{equation}
with $u = |w_\alpha|$, where, as done in the previous case, we compute 
$$\frac{1}{2|p|}e_\theta  \cdot \nabla_w G_p \partial_\alpha |w_\alpha|^2 = \alpha |p| e_\theta  \cdot \nabla_w G_p - |q_\alpha| e_p  \cdot \nabla_w G_p = - |p| |q_\alpha| \frac{1+2|w+p|^2}{\cE(w+p)} .
$$
Combining, we obtain 
\begin{equation}\label{dS-Hprad11}
\frac{d\sigma(w)}{|\nabla_wG_p|} = \frac{\cE(w+p) u du d\theta}{|p|(1+2|w+p|^2)} .\end{equation}
Recalling $\cE(w) = |w|\sqrt{1+|w|^2}$, we have 
$$\frac{\cE(w+p)|w|}{|p|(1+2|w+p|^2)} \le |w||p|^{-1}.$$
On the other hand, by considering $|p|\le 1$ and $|p|\ge 1$ and using the fact that $|w|+|p|\le 2|w+p|$ (on $S_p'$), we have 
$$\frac{\cE(w+p)|w|}{|p|(1+2|w+p|^2)} \ge c_0 |w| \min\{ 1,|p|^{-1}\}.$$
This yields the upper and lower bounds on the surface integral. \end{proof}

\section{Moment estimates}\label{sec-moments}

In this section, we shall derive estimates on the energy moment and on the mass of nonnegative solutions of \eqref{QuantumBoltzmann}. In order to obtain the boundedness of the gain and loss terms, which is crucial in the proof of the main theorem, we are obliged to bound the second-order energy moment (Proposition \ref{Propo:SecondOrderMoment}). In what follows, we take initial data $f_0(p)=f_0(|p|)$ with finite energy, and thus thanks to the conservation of energy \eqref{Coro:ConservatioMomentum}, energy remains finite for all times.


\begin{proposition}\label{Propo:SecondOrderMoment} Let $f_0(p)=f_0(|p|)\ge 0$ have finite energy. Then, for any $\tau>0$, nonnegative radial solutions $f(t,p) = f(t,|p|)$ of \eqref{QuantumBoltzmann} with initial data $ f_0(p)$ satisfy
\begin{equation}\label{2nd-moment}\sup_{t\in[\tau,\infty)}\int_{\mathbb{R}^3}f(t,p)\mathcal{E}^2(p)dp<+\infty.\end{equation}
\end{proposition}
\begin{proof}
Take $\varphi =\mathcal{E}^2(p)$ to be the test function in Lemma \ref{Lemma:WeakFormulation}. We obtain 
$$
\begin{aligned}
\frac{d}{dt}
 \int_{\mathbb{R}^3} f\mathcal{E}^2dp ~~ = & ~~ \iiint_{\mathbb{R}^9} \mathcal{R}_{p, p_1, p_2}[f] \Big( \mathcal{E}^2(p)-\mathcal{E}^2(p_1)-\mathcal{E}^2(p_2)\Big)dpdp_1dp_2.
\end{aligned}
$$
In view of the Dirac delta functions in the collision kernel \eqref{Kernel}, the integral is on the surface dictated by the conditions $p = p_1 + p_2$ and $\cE(p) = \cE(p_1) + \cE(p_2)$. In particular, on the surface, $\mathcal{E}^2(p)-\mathcal{E}^2(p_1)-\mathcal{E}^2(p_2) = 2\mathcal{E}(p_1)\mathcal{E}(p_2)$. Thus, upon recalling that $f\ge 0$, we have 
$$
\begin{aligned}
\frac{d}{dt} \int_{\mathbb{R}^3} f\mathcal{E}(p)^2dp 
& = 2 \iiint_{\mathbb{R}^9} \mathcal{R}_{p, p_1, p_2}[f] \mathcal{E}(p_1)\mathcal{E}(p_2) dpdp_1dp_2
\\&=  2\iiint_{\mathbb{R}^9}\mathcal{K}(p,p_1,p_2) \Big( f_1f_2-(1+f_1+f_2)f \Big )\mathcal{E}(p_1)\mathcal{E}(p_2)dpdp_1dp_2
\\&\le  2\iiint_{\mathbb{R}^9}\mathcal{K}(p,p_1,p_2) \Big( f_1f_2- f \Big )\mathcal{E}(p_1)\mathcal{E}(p_2)dpdp_1dp_2.
\end{aligned}
$$
Let us set 
\begin{equation}\label{def-J1234}\begin{aligned}
J_1 :& = 2\iiint_{\mathbb{R}^9}\mathcal{K}(p,p_1,p_2) f_1 f_2 \mathcal{E}(p_1)\mathcal{E}(p_2)dpdp_1dp_2
\\
J_2: &= - 2\iiint_{\mathbb{R}^9}\mathcal{K}(p,p_1,p_2) f \mathcal{E}(p_1)\mathcal{E}(p_2)dpdp_1dp_2.\end{aligned}
\end{equation}
We write $J_1,J_2$ in term of surface integrals. Recalling 
$$\begin{aligned}
\mathcal{K}(p, p_1, p_2)&=|p|^\rho|p_1|^\rho|p_2|^\rho \delta (p-p_1-p_2)\delta (\mathcal{E}(p) -\mathcal{E}(p_1)-\mathcal{E}(p_2)) 
\end{aligned}
$$
and following \eqref{intro-Surface123}, we write 
$$
\begin{aligned}J_1 & = 2\iint_{\mathbb{R}^6}\mathcal{K}(p_1+p_2,p_1,p_2) f_1f_2  \mathcal{E}(p_1)\mathcal{E}(p_2)dp_1dp_2
\\
&\lesssim \int_{\RR^3}\int_{S_{p_1}'} |p_1+p_2|^\rho|p_1|^\rho|p_2|^\rho f_1f_2  \mathcal{E}(p_1)\mathcal{E}(p_2) \frac{d\sigma(p_2)dp_1}{|\nabla G_{p_1}(p_2)|}
\\
&\lesssim \int_{\RR^3}  |p_1|^\rho f_1  \mathcal{E}(p_1) \Big(  \int_{S_{p_1}'} |p_2|^\rho (|p_1|^\rho+|p_2|^\rho) f_2 \mathcal{E}(p_2) \frac{d\sigma(p_2)}{|\nabla G_{p_1}(p_2)|} \Big) dp_1.
\end{aligned}$$
By Lemma \ref{lem-Sp1}, and the fact that $f$ is radial, the surface integral is estimated by 
$$
\begin{aligned} \int_{S_{p_1}'}(|p_1|^\rho+|p_2|^\rho)|p_2|^\rho f_2\mathcal{E}(p_2)  \frac{d\sigma(p_2)}{|\nabla G_{p_1}(p_2)|} 
&\lesssim  C|p_1|^{-1}\int_{\mathbb{R}_+}(|p_1|^\rho+|p_2|^\rho)|p_2|^{\rho+1}f_2\mathcal{E}(p_2)d(|p_2|)\\
&\lesssim C|p_1|^{-1}\int_{\mathbb{R}^3}(|p_1|^\rho+|p_2|^\rho)|p_2|^{\rho-1}f_2\mathcal{E}(p_2)dp_2 .
\end{aligned}
$$
Thus, upon recalling $\mathcal{E} (p)=|p|\sqrt{1+|p|^2} \ge |p|$, we obtain 
$$
\begin{aligned}J_1 
&\lesssim  \iint_{\RR^6} (|p_1|^\rho + |p_2|^\rho)|p_2|^{\rho-1} f_1  \mathcal{E}(p_1) f_2 \mathcal{E}(p_2) 
dp_1 dp_2
\\
&\lesssim  \iint_{\RR^6}  (\mathcal{E}(p_1)^{\rho}+\mathcal{E}(p_2)^{\rho}) f_1  \mathcal{E}(p_1) f_2 \mathcal{E}(p_2)^{\rho} 
dp_1 dp_2
\\
&\lesssim \Big( \int_{\RR^3}  \mathcal{E}(p_1)^{1+\rho} f_1 \; dp_1 \Big)  \Big( \int_{\RR^3}  \mathcal{E}(p_2)^{\rho} f_2 \; dp_2 \Big) + \Big( \int_{\RR^3}  \mathcal{E}(p_1) f_1 \; dp_1 \Big)  \Big( \int_{\RR^3}  \mathcal{E}(p_2)^{2\rho} f_2 \; dp_2. \Big),\end{aligned}
$$

By using the fact that the mass, energy, the quantity $\int_{\RR^3}  \mathcal{E}(p_2)^{\rho} f_2 \; dp_2 $ are bounded and $$2\mathcal{E}(p_1)^{1+\rho}\le \mathcal{E}(p_1)^2+\mathcal{E}(p_1)^{2\rho}$$ the above yields 
\begin{equation}\label{est-J11}
\begin{aligned}J_1~~~
&\lesssim  \int_{\RR^3}  \mathcal{E}(p_1)^{1+\rho} f_1 \; dp_1 + \int_{\RR^3}  \mathcal{E}(p_2)^{\rho} f_2 \; dp_2\\
&\lesssim  \int_{\RR^3}  \mathcal{E}(p_1)^{2} f_1 \; dp_1 + \int_{\RR^3}  \mathcal{E}(p_2)^{\rho} f_2 \; dp_2. 
\end{aligned}
\end{equation}

Next, we estimate the integral $J_2$ in \eqref{def-J1234}. Following \eqref{intro-Surface123}, we estimate  
$$
\begin{aligned}J_2 
& =- 2\iint_{\mathbb{R}^6}\mathcal{K}(p,p_1,p-p_1) f\mathcal{E}(p_1)\mathcal{E}(p-p_1)dpdp_1
\\
& = - 2 \int_{\RR^3} \Big( \int_{S_{p}} |p_1|^{\rho}|p-p_1|^{\rho}  \mathcal{E}(p_1)\mathcal{E}(p-p_1)\frac{d\sigma(p_1)}{|\nabla H_p(p_1)|}\Big) |p|^{\rho} fdp 
.\end{aligned}$$
Recalling $\cE(p)\ge \frac12(|p| + |p|^2)$ and using \eqref{lem-Sp-e2} in Lemma \ref{lem-Sp}, we estimate 
\begin{equation}\label{low-J2S}
\begin{aligned}
& |p|^{\rho} \int_{S_{p}} |p_1|^{\rho}|p-p_1|^{\rho}  \mathcal{E}(p_1)\mathcal{E}(p-p_1)\frac{d\sigma(p_1)}{|\nabla H_p(p_1)|}\\
\gtrsim &   (|p|^{3\rho+3} + |p|^{3\rho+5})\;  \min\{ 1, |p|\}
\gtrsim  |p|^{3\rho+5}.
 \end{aligned}\end{equation}
 This proves 
\begin{equation}\label{Propo:SecondOrderMoment:E3a}
\begin{aligned}
J_2 \lesssim &- \int_{\mathbb{R}^3}|p|^{3\rho+5} fdp .
\end{aligned}
\end{equation}

Moreover, we also have
 $$ 
\begin{aligned}
\int_{\RR^3} \cE^{2\rho}(p) f \; dp 
\lesssim  
C(\epsilon) \Big( 
\int_{\RR^3} \cE^2(p) f\; dp\Big)
 + \epsilon\left(\int_{\mathbb{R}^3}|p|^{3\rho+5} fdp\right),
   \end{aligned}$$ 
   for some small constant $\epsilon$.

Therefore 

\begin{equation}\label{Propo:SecondOrderMoment:E3}
\begin{aligned}
J_2'& = J_2 + C(\epsilon) \Big( 
\int_{\RR^3} \cE^2(p) f\; dp\Big)
 + \epsilon\left(\int_{\mathbb{R}^3}|p|^{3\rho+5} fdp\right)\\
  \lesssim &- \int_{\mathbb{R}^3}|p|^{3\rho+5} fdp +  \int_{\RR^3} \cE^2(p) f\; dp.
\end{aligned}
\end{equation}

In addition, using the H\"older inequality, we estimate 
$$ 
\begin{aligned}
\int_{\RR^3} \cE^2(p) f \; dp 
&\lesssim  
\int_{\RR^3} |p|^2 f\; dp
 +  \int_{\RR^3} |p|^4 f\; dp 
\\
&\lesssim C 
\int_{\RR^3} |p|^2 f\; dp
 +  
 \left(\int_{\mathbb{R}^3}|p|^{3\rho+5} fdp\right)^{\frac{2}{3\rho+3}}\left(\int_{\mathbb{R}^3}|p|^2 fdp\right)^{\frac{3\rho+1}{3\rho+3}} 
\\
&\lesssim  
\Big(\int_{\RR^3} |p|^2 f\; dp\Big)^{\frac{2}{3\rho+3}}
 +  
 \left(\int_{\mathbb{R}^3}|p|^{3\rho+5} fdp\right)^{\frac{2}{3\rho+3}} ,\end{aligned}$$ 
in which the last inequality was due to the fact that $|p|^2\le \cE(p)$ and the energy is bounded. Again, using $|p|\le \cE(p)$, we thus obtain 
 $$ 
\begin{aligned}
\int_{\RR^3} \cE^2(p) f \; dp 
\lesssim  
\Big( 
\int_{\RR^3} \cE^2(p) f\; dp\Big)^{\frac{2}{3\rho+3}}
 + \left(\int_{\mathbb{R}^3}|p|^{3\rho+5} fdp\right)^{\frac{2}{3\rho+3}} .
   \end{aligned}$$

This and \eqref{Propo:SecondOrderMoment:E3} yield
\begin{equation}\label{est-J22} J_2' \lesssim 
- \theta_0  \Big( \int_{\RR^3} \cE^2(p) f \; dp  \Big)^{\frac{3\rho+3}{2}} + \int_{\RR^3} \cE^2(p) f\; dp .\end{equation}

To conclude, we have proven 
\begin{equation}
\label{Propo:SecondOrderMoment:E6}
\frac{d}{dt} \int_{\mathbb{R}^3} f\mathcal{E}^2dp  \lesssim \int_{\mathbb{R}^3} f\mathcal{E}^2dp \Big[ 1 -\theta_1 \Big(\int_{\mathbb{R}^3}f\mathcal{E}^2dp\Big)^{\frac{3\rho+3}{2}} \Big]
\end{equation}
for some positive constants $C_1, \theta_1$. Since $f\ge 0$, the standard ODE argument applying to the differential inequality \eqref{Propo:SecondOrderMoment:E6} yields at once 
the boundedness of $\int_{\mathbb{R}^3} f\mathcal{E}^2dp$; for instance, there holds 
$$  \int_{\mathbb{R}^3} f(t,p)\mathcal{E}^2dp \lesssim \max \Big\{ \frac{1}{{\theta_1}^{\frac{2}{3\rho+3}}},  \int_{\mathbb{R}^3} f(\tau,p)\mathcal{E}^2dp \Big\}$$ 
for all $t \ge \tau$. The proposition follows. 
\end{proof}

\begin{remark}
Following similar lines of the above proof, we can in fact show that energy moments at any order are created and propagated in positive times as in Proposition \ref{Propo:SecondOrderMoment} for the second-energy moment.  
We skip the details as the result will not be used in this paper.\end{remark}



 \section{Uniform lower bound}


In this section, we shall prove our main theorem, Theorem \ref{TheoremLowerBound}. Let us first write the collision operator as follows:
\begin{equation}\label{def-gainloss} 
\mathcal{C}_{12}[f] = Q_\mathrm{gain}[f]  -  Q_\mathrm{loss}[f]\end{equation}
where the Gain and Loss operators are defined by 
$$
\begin{aligned}
Q_\mathrm{gain}[f]&:=\iint_{\mathbb{R}^3\times\mathbb{R}^3} \mathcal{K}(p, p_1, p_2) f_1 f_2 
dp_1dp_2 + 2\iint_{\mathbb{R}^3\times\mathbb{R}^3}\mathcal{K}(p_1, p, p_2) (1+f + f_2) f_1 \; dp_1 dp_2
\\
Q_\mathrm{loss}[f]&:= f\iint_{\mathbb{R}^3\times\mathbb{R}^3} \mathcal{K}(p, p_1, p_2) (1 + 2 f_2)
dp_1dp_2 + 2f\iint_{\mathbb{R}^3\times\mathbb{R}^3}\mathcal{K}(p_1, p, p_2) f_2 \; dp_1 dp_2.
\end{aligned}$$
For convenience, we also write 
\begin{equation}\label{def-LfQ}Q_\mathrm{loss}[f] = f \mathcal{L}[f],\end{equation}
$\mathcal{L}[f]$ is usually called the collision frequency.
\bigskip

\begin{lemma}\label{LemmGainLoss}
Suppose that $F(p)\leq G(|p|)$, for some radially symmetric function $G$ with $$\mathcal{M} = \int_{\mathbb{R}_+}G(u) \; (u^{1+\rho} + u^{1+2\rho}) du  < \infty.$$
Then, there holds
\begin{eqnarray}\label{LossBound}
\mathcal{L}[F](p) \le C_0 \mathcal{M} (|p|^\rho+|p|^{2\rho}) + C_0 |p|^{3\rho+1}
\end{eqnarray}
for some positive universal constant $C_0$. 

In addition, if 
$$\mathcal{N} = \int_{\mathbb{R}_+}G(u) \; u^{\rho+1}  du  < \infty,$$
and
$$\mathcal{P} = \int_{\mathbb{R}_+}G(u) \; u^{2\rho+1}  du  < \infty,$$
then 
\begin{eqnarray}\label{LossBound2}
\mathcal{L}[F](p) \le C_1 \mathcal{P}|p|^\rho + C_1\mathcal{N}|p|^{2\rho}+ C_1 |p|^{3\rho+1}
\end{eqnarray}
for some positive universal constant $C_1$. 

\end{lemma}
\begin{proof} We first write the collision integrals in term of surface integrals. Following \eqref{intro-Surface123}, we have 
$$
\begin{aligned}
\mathcal{L}[F] = &\int_{S_p} |p|^\rho| p - p_2|^\rho| p_2|^\rho (1 + 2F_2 ) \; \frac{d\sigma(p_2)}{|\nabla H_p(p_2)|}
+ 2\int_{S_p'}|p + p_2|^\rho|p|^\rho| p_2|^\rho F_2 \; \frac{d\sigma(p_2)}{|\nabla G_p(p_2)|}.
\end{aligned}$$
Consider the surface integral over $S_p$. Recall that that  $|p_2|\le |p|$ and $|p-p_2|\le |p|$ on $S_p$. Hence, using Lemma \ref{lem-Sp}, we estimate 
$$\begin{aligned}
&\int_{S_p} |p|^\rho| p - p_2|^\rho| p_2|^\rho  (1 + 2F_2) \; 
\frac{d\sigma(p_2)}{|\nabla H_p(p_2)|}\\
&\lesssim   |p|^{2\rho}  \int_{S_p}  (1 + 2G(|p_2|) )|p_2|^\rho \; \frac{d\sigma(p_2)}{|\nabla H_p(p_2)|}\\
&\lesssim |p|^{2\rho}  \int_0^{|p|} (1+ G(u)) \; \min\{1,u\} u^{\rho} du 
\\
&\lesssim |p|^{3\rho+1}  +  |p|^{2\rho} \int_0^{|p|} G(u)u^{\rho+1}  du ,
\end{aligned}$$
which yields the claimed bound for the integral on $S_p$. Next, we check the integral on $S_p'$. Lemma \ref{lem-Sp1} yields 
$$\begin{aligned}
\int_{S_p'}|p + p_2|^{\rho}|p|^{\rho}| p_2|^{\rho}F_2 \; \frac{d\sigma(p_2)}{|\nabla G_p(p_2)|}
&\lesssim \int_{S_p'} (|p|^{\rho} + |p_2|^{\rho})|p|^{\rho}| p_2|^{\rho} G(|p_2|) \; \frac{d\sigma(p_2)}{|\nabla G_p(p_2)|}
\\
&\lesssim |p|^{\rho}\int_0^\infty (|p|^{\rho} + u^{\rho}) G(u) u^{\rho+1} du
\end{aligned} $$ 
which is bounded by $C_0 \mathcal{M} (|p|^{2\rho}+|p|^{\rho}) $ and $C_1 \mathcal{P}|p|^{\rho} + C_1\mathcal{N}|p|^{2\rho}$. The lemma follows. 
\end{proof}

\begin{lemma}\label{LemmaGainLowerBound} Let $\delta, \theta >0$, and $F$ be any nonnegative smooth function so that $F(p)\geq\theta$ on $B_\delta: = \{|p|\leq\delta\}$. Then, there exists a universal constant $c_0>0$ such that
\begin{equation}
Q_\mathrm{gain}[F](p)\geq c_0 |p|^{3\rho+1}\min\{1,|p|\} \theta^2
\end{equation}
for all $p\in B_{\sqrt2 \delta }.$
\end{lemma}
\begin{proof} By definition \eqref{def-gainloss} and the assumption on the lower bound on $F$, we have 
$$ 
\begin{aligned}
Q_\mathrm{gain}[F](p)&=\int_{S_p} {K}(p, p - p_2, p_2)F(p - p_2)F(p_2) \; d\sigma(p_2)
\\&\quad + 2\int_{S_p'}K(p + p_2, p, p_2)F(p + p_2) \Big (F(p)+F(p_2)+1\Big )\; d\sigma(p_2)
\\&\gtrsim \int_{S_p} \mathcal{K}(p, p - p_2, p_2)F(p - p_2)F(p_2) \; d\sigma(p_2)
\\&\gtrsim  |p|^{\rho} \theta^2\int_{S_p \cap B(0,\delta) \cap B(p,\delta)} |p-p_2|^{\rho}|p_2|^{\rho}d\sigma(p_2) ,
\end{aligned} $$
in which we note again that $p_2, p-p_2$ are both in $B_\delta$, thanks to the monotonicity of the energy function $\mathcal{E}(p)$.

To proceed, we consider three cases. First, take $p \in B(0,\delta)\setminus B(0, \frac{\delta}{2})$. In this case, $B(\frac p2, \frac{|p|}{2}) \subset B(0,\delta) \cap B(p,\delta)$, and so we can estimate
$$ 
\begin{aligned}
Q_\mathrm{gain}[F](p)
&\gtrsim  |p|^{\rho} \theta^2 \int_{S_p \cap B(\frac p2, \frac{|p|}{2})} |p-p_2|^{\rho}|p_2|^{\rho} d\sigma(p_2)
\\
&\gtrsim |p|^{3\rho+1}\min\{1,|p|\} \theta^2 ,
\end{aligned}$$ 
for some positive constants $c_0,c_1$, thanks to the lower bound \eqref{lem-Sp-e2} in Lemma \ref{lem-Sp}, with $\gamma=1$.

Next, for $p\in B(0,\frac{\delta}{2})$, we note that $B(0, \frac\delta 2) \subset B(0,\delta) \cap B(p,\delta)$. Hence, in this case, we have, by the lower bound \eqref{lem-Sp-e2}, 
$$ 
\begin{aligned}
Q_\mathrm{gain}[F](p)
&\gtrsim  |p|^{\rho} \theta^2 \int_{S_p \cap B(0,\frac\delta 2)} |p-p_2|^{\rho} |p_2|^{\rho} d\sigma(p_2) 
\\&\gtrsim |p|^{3\rho+1}\min\{1,|p|\} \theta^2.
\end{aligned}$$ 
The lemma is proved for $|p|\le \frac \delta 2$. 

Finally, we consider  the case when $p \in B(0,\sqrt 2 \delta) \setminus B(0,\delta)$. In this case, we check that $S_p \cap B(0,\delta) \cap B(p,\delta)$ has positive surface area. Indeed, let $D_p$ be the disk that is centered at $\frac p2$, of radius $\sqrt{\delta^2 - \frac{|p|^2}{4}}$, and is on the plane orthogonal to $p$. Let $x$ be a point on the boundary of $D_p$, then $|x-p/2|=\sqrt{\delta^2 - \frac{|p|^2}{4}}$ and  $x-p/2$ is orthogonal to $p$. As a consequence, $|x|^2=|x-p/2|^2+|p/2|^2=\delta^2$ and $|x-p|^2=|x-p/2|^2+|p/2|^2$.  It is clear that $D_p$ belongs to the intersection $B(0,\delta) \cap B(p,\delta)$ and, since $\sqrt{\delta^2 - \frac{|p|^2}{4}} \ge \frac{|p|}{2}$, the surface $S_p$ crosses the interior of $D_p$. This proves that $S_p \cap B(0,\delta) \cap B(p,\delta)$ is non empty. Since $B(0,\delta) \cap B(p,\delta)$ has positive Lebesgue measure, the surface area of $S_p \cap B(0,\delta) \cap B(p,\delta)$ is bounded below from zero by a constant times $|p|$, since any geodesic on the surface starting from $0$ to $p$ has a greater length than $|p|$. We can then compute 
$$ 
\begin{aligned}
Q_\mathrm{gain}[F](p)
&\gtrsim  |p|^{\rho} \theta^2\int_{S_p \cap B(0,\delta) \cap B(p,\delta)} |p-p_2|^{\rho} |p_2|^{\rho} d\sigma(p_2)
\\&\gtrsim |p|^{3\rho+1}\min\{1,|p|\} \theta^2,
\end{aligned}$$ 
due  to the lower bound \eqref{lem-Sp-e2}.
This completes the proof of the lemma. 
\end{proof}

\begin{lemma}\label{LemmSolutionLowerBound}
Let $\delta, \theta>0$. Suppose that initial data $f_0(p)\geq\theta$ on $B_\delta $,  where  $B_\delta=\{|p|\leq\delta\}$. Let $f(t,p)$ be a solution to \eqref{QuantumBoltzmann} so that $f(t,p)\leq G(t,|p|)$ for all $t\ge 0$ and for some radially symmetric function $G$ so that 
\begin{equation}\label{moments-G}\mathcal{M} (t)= \sup_{0\le s\le t} \int_{\mathbb{R}_+}G(s,u) ( | u |^{\rho+1} + |u|^{2\rho+1})du <\infty .\end{equation}
Then, there holds the following uniform lower bound 
\begin{equation}\label{LemmSolutionLowerBound:1}
f(t,p) \geq C_0{t} e^{- t   \mathcal{M}(t)L_*(\delta) }
|p|^{3\rho+1}\min\{1,|p|\} \theta^2, \qquad \forall~t\ge 0,
\end{equation}
for all $p\in B_{{\sqrt2\delta}}$, $$L_*(\delta) :=  c_0( 1+ \delta^{3\rho+1}).$$ Here, $c_0, C_0$ are some universal positive constants  independent of $\mathcal{M}$, $\delta, \theta$ and $p$. 

In addition, if 
$$\mathcal{N}(t)= \sup_{0\le s\le t} \int_{\mathbb{R}_+}G(s,u) \; u^{\rho+1}  du  < \infty,$$
and
$$\mathcal{P}(t) =\sup_{0\le s\le t} \int_{\mathbb{R}_+}G(t,u) \; u^{2\rho+1}  du  < \infty$$
is conserved for all time $t$,
then \begin{equation}\label{LemmSolutionLowerBound:2}
f(t,p) \geq C_1{t} e^{- t   [\mathcal{N}(t)L^*(\delta)+c_1\mathcal{P}(t)\delta^\rho]}
|p|^{3\rho+1} \min\{1,|p|\} \theta^2, \qquad \forall~t\ge 0,
\end{equation}
for all $p\in B_{{\sqrt2\delta}}$, $$L^*(\delta) :=  c_1(1 + \delta^{3\rho+1}).$$ Here, $c_1, C_1$ are some universal positive constants  independent of $\mathcal{N}$, $\mathcal{P}$, $\delta, \theta$ and $p$. 
\end{lemma}
\begin{proof} Using Lemma \ref{LemmGainLoss}, with $F = f(t,p)$, we obtain 
\begin{equation}\label{Low-ineq}
\partial_t f (t,p) + L_0(t,|p|) f(t,p) \ge Q_\mathrm{gain}[f](t,p) 
\end{equation}
with $L_0(t,|p|) = C_0 \mathcal{M}(t) (1+|p|^{2\rho} ) + C_0 |p|^{3\rho+1} $. Note that $\mathcal{M}(t)$ and hence $L_0(t,|p|)$ are increasing in $t$. Using the monotonicity and applying the Duhammel's representation to \eqref{Low-ineq}, we obtain 
\begin{equation}\label{Duh}
f(t,p)
\geq f_0(p) e^{- \int_0^tL_0(s,|p|)ds}
+\int_0^t e^{-\int_\tau^t L_0(s,|p|)ds}Q_\mathrm{gain}[f](\tau, p)d\tau
\end{equation}
for all $t\ge 0$. Since $Q_\mathrm{gain}[f](p) \ge 0$ and $L_0(t, |p|)$ is an increasing function in $t$, it follows that for $p \in B_\delta$, \eqref{Duh} yields 
\begin{equation}\label{f-lowww}
f(t,p)\quad \geq \quad f_0(p) e^{-t L_0(t,|p|)} \quad \ge\quad  \theta e^{-t L_0(t,\delta)}, \qquad t\ge 0.
\end{equation}
Next, for each fixed time $t\ge 0$, we now apply Lemma \ref{LemmaGainLowerBound} for $F = f(t,p)$, with the new lower bound \eqref{f-lowww} on $B_\delta$, 
yielding 
$$Q_\mathrm{gain}[f](t,p) \ge  C_0|p|^{3\rho+1} \min\{1,|p|\}\theta^2 e^{-2 t L_0(t,\delta)},
$$
for all $p \in B_{\sqrt 2 \delta}$. Putting this into \eqref{Duh}, we obtain  
$$
\begin{aligned}
f(t,p)
&\geq \int_0^t e^{-\int_\tau^t L_0(s,|p|)ds} Q_\mathrm{gain}[f](\tau, p)d\tau
\\&\geq \int_0^t e^{-\int_\tau^t L_0(t,|p|)ds} Q_\mathrm{gain}[f](\tau, p)d\tau
\\&\geq \int_0^t e^{-(t-\tau) L_0(t,\delta) } Q_\mathrm{gain}[f](\tau, p)d\tau
\\&\gtrsim |p|^{3\rho+1} \min\{1,|p|\} \theta^2 \int_0^t e^{-(t-\tau) L_0(t,\delta) } e^{-2 \tau L_0(t,\delta)} d\tau
\\&\gtrsim|p|^{3\rho+1} \min\{1,|p|\} \theta^2 e^{-2 t  L_0(t,\delta)}  t
\\&\gtrsim|p|^{3\rho+1} \min\{1,|p|\} \theta^2  e^{-t C_0 \mathcal{M}(t)  L_*(\delta)} t.
\end{aligned}
$$
This completes the proof of \eqref{LemmSolutionLowerBound:1}. The second inequality \eqref{LemmSolutionLowerBound:2} can be proved by exactly the same procedure.
\end{proof}

\subsection{Proof of Theorem \ref{TheoremLowerBound}} 

We are now ready to give the proof of Theorem \ref{TheoremLowerBound}. Let $\theta_0, R_0>0$ as in the assumption of Theorem \ref{TheoremLowerBound} so that $f_0(p)\geq 2\theta_0 $ on $B_{2R_0}=\{|p|\leq 2R_0\}$. Let $\tau$ be sufficiently small so that $f(\tau,p)\geq\theta_0 $ on $B_{R_0}$, thanks to the continuity in time of the (classical) solution $f(t,p)$.

In the proof, we shall apply Lemma \ref{LemmSolutionLowerBound} repeatedly to the solution $f(t,p)$ of \eqref{QuantumBoltzmann}, with $G(t,|p|) = f(t,|p|)$. First, we note that since $f(t,p)$ is radially symmetric and $\cE(p)\ge |p|$, we have  
$$
\begin{aligned}
 \int_{\mathbb{R}_+}f(t,|p|) (|p|^{1+\rho}+|p|^{1+2\rho})  d|p| \le C_0 \int_{\RR^3} f(t,p) (1+\cE(p)^2)\; dp \le C_\tau,\end{aligned}$$
for all $t\ge \tau$, thanks to the conservation of mass and the boundedness of second order energy-moment. This verifies the assumption \eqref{moments-G} on $G(t,|p|) = f(t,|p|)$, made in Lemma  \ref{LemmSolutionLowerBound}, with $\mathcal{M}(t) = C_\tau$, which is time-independent.

Fix a positive and sufficiently small $\delta <  R_0$, and a positive time $t_0$ so that 
\begin{equation} t_0<\frac14.\end{equation}
Since $f_0(\tau,p)\ge \theta_0$ on $B_\delta$, applying Lemma  \ref{LemmSolutionLowerBound}  to the solution $f(t,p)$ of \eqref{QuantumBoltzmann} with the initial data $f(\tau,p)$ yields 
\begin{equation}\label{TheoremLowerBound:1}\begin{aligned}
f(\tau+t_0,p) &\geq  t_0 e^{- t_0C_\tau L_*(\delta) }  C_p \theta_0^2, 
\end{aligned} 
\end{equation}
for all $p \in B_{\sqrt 2 \delta}$, in which $L_*(\delta) = c_0( 1+ \delta^{3\rho+1}) $, 
\begin{equation}\label{def-Cppp}C_p := C_0  |p|^{3\rho+1} \min\{1,|p|\}.\end{equation}

We stress that $C_p$ does not depend on $\delta$ and $t_0$, and hence the estimate \eqref{TheoremLowerBound:1} can be iterated. Indeed, applying again Lemma \ref{LemmSolutionLowerBound} to the solution $f(t,p)$ of \eqref{QuantumBoltzmann} with the initial data $f(\tau+t_0,p)$ satisfying \eqref{TheoremLowerBound:1}, yielding

$$ 
\begin{aligned}
f(\tau+ t_0+t_1,p) 
&\ge  t_1 e^{- t_1 L_*(\sqrt 2\delta) }
C_p  \Big[  t_0 e^{- t_0 L_*(\delta) }
C_p\theta_0^2 \Big]^2
\\
&\ge 
 t_1  t_0^2 e^{- t_1 L_*(\sqrt 2 \delta)} e^{-2t_0 L_*(\delta) } 
\Big(C_p \Big)^{1+2} \theta_0^{2^2}
\end{aligned}$$
for arbitrary positive time $t_1<\frac14$ and for all $p\in B_{{\sqrt2^2}\delta}$. For each fixed integer $n\ge 2$, we iteratively apply Lemma \ref{LemmSolutionLowerBound}, yielding 
$$
\begin{aligned}
f(\tau+ t_0+\cdots + t_n ,p) 
& \ge 
t_n t_{n-1}^2\cdots t^{2^k}_{n-k}\cdots  t_0^{2^n} e^{- t_n  L_*(\sqrt{2}^n \delta)}\cdots e^{-2^n t_0 L_*(\delta) }  
\\& \quad \times 
\Big(C_p\Big)^{1+ 2+\cdots + 2^n} \theta_0^{2^{n+1}},
\end{aligned}
$$
for all $p \in B_{\sqrt{2}^{n+1} \delta}$. By using $1+ 2+\cdots + 2^n = 2^{n+1} - 1$, the above is reduced to 
\begin{equation}\label{bd-tn} 
\begin{aligned}
f(\tau+t_0+\cdots + t_n ,p) 
& \ge 
 t_n t_{n-1}^2 \cdots t^{2^k}_{n-k}\cdots  t_0^{2^n} \theta_0
\Big( C_p \theta_0\Big)^{2^{n+1}-1} E_n,
\end{aligned}\end{equation}
for all $p \in B_{\sqrt{2}^{n+1} \delta}$, in which for convenience we have set 
\begin{equation}\label{def-Ennn} E_n : = e^{- t_n L_*(\sqrt{2}^n \delta)}\cdots e^{-2^k t_{n-k}  L_*(\sqrt{2}^{n-k}\delta) }\cdots e^{-2^n t_0 L_*(\delta) } . \end{equation}



%

~\\
{\bf Case 1: $|p|> \sqrt 2\delta$.} Recall that $\delta, t_0$ are fixed.  For each $p$ so that $|p|> \sqrt 2 \delta$, we take an integer $n$ satisfying  
\begin{equation}\label{psqrt2}\sqrt{2}^{n} \delta< |p| \le \sqrt{2}^{n+1} \delta.\end{equation}
In particular, $p \in B_{\sqrt{2}^{n+1} \delta}$ and \eqref{bd-tn} holds for arbitrary positive time steps $t_k$. 
We now fix an arbitrary time $t\in (\tau, t_*)$, with $t_* = 1/4$. We take $ t_k = t_0^k$ and choose $t_0$ so that $t_0 < \frac14$ and 
$$ \sum_{k=0}^n  t_k = t.$$
Such a choice of $t_0$ is possible by the definition of $t_*$. The lower bound \eqref{bd-tn} then reads
\begin{equation}\label{bd-tn11} 
\begin{aligned}
f(\tau+ t ,p) 
& \ge 
\theta_0
 t_n t_{n-1}^2\cdots t^{2^k}_{n-k}\cdots  t_0^{2^n}  \Big(C_p \theta_0\Big)^{2^{n+1}-1} E_n,
\end{aligned}\end{equation}
for all $t\in (\tau,t_*)$ and all $|p|> \sqrt 2 \delta$, with $n$ being defined by \eqref{psqrt2}.  

Note in particular that $t_0 \ge T_\tau$ for some positive time $T_\tau$, since $t\ge \tau$. Using this, we can estimate 
$$\begin{aligned}
 t_n t_{n-1}^2\cdots t^{2^k}_{n-k}\cdots  t_0^{2^n} &\ge T_\tau^{n+ 2(n-1)+\cdots + 2^k (n-k)+\cdots+2^n}  
\\
&\ge 
T_\tau^{2^n + \sum_{k=0}^n 2^k (n-k)}  
\\& \ge T_\tau^{2^n (1+ \sum_{k=0}^\infty k 2^{-k})}  
\\
& = \mathcal{C}_0^{2^n}
\end{aligned}$$
in which $\mathcal{C}_0 = T_\tau^{1+ \sum_{k=0}^\infty k 2^{-k}}$, which is finite and nonzero. 

Next, by the definition \eqref{def-Cppp} of $C_p$, we have $C_p \ge C_{\delta}$ for some positive constant $C_{\delta}$, since $|p|> \sqrt 2 \delta$, and hence 
$$
\begin{aligned}
\theta_0
\Big( C_p\ \theta_0\Big)^{2^{n+1}-1} 
&\ge \theta_0 \Big( C_{\delta} \theta_0\Big)^{2^{n+1}-1} 
\geq\mathcal{C}_1(\mathcal{C}_2)^{2^n}
\end{aligned}$$
for some positive constants $\mathcal{C}_1$ and $\mathcal{C}_2$, independent of $n, p$ and $t$.

Finally, we estimate the exponential term $E_n$ defined as in \eqref{def-Ennn}. Recalling that $4 t_0<1$, $t_k = t_0^k$ and $L_*(\delta) = c_0 (1 + \delta^{3\rho+1})$, we have 
$$
\begin{aligned}
e^{-2^k t_{n-k} L_*(\sqrt{2}^{n-k}\delta) } 
&\ge e^{-2^{k} t_0^{n-k} c_0 (1 + \sqrt{2}^{4(n-k)}\delta^{3\rho+1}) } 
\\
&\ge e^{-2^{k}c_0[t_0^{n-k} +  (4t_0)^{n-k}\delta^{3\rho+1}]}
\\
& = e^{-2^{k}c_0[1 +  \delta^{3\rho+1}]}.
\end{aligned}$$
Hence, we obtain $$
\begin{aligned}
 E_n  &= \exp \Big(- \sum_{k=0}^n 2^k t_{n-k} L_*(\sqrt{2}^{n-k}\delta) \Big) 
 \\
 &\ge  \exp \Big(- c_0[1 +  \delta^{3\rho+1}]\sum_{k=0}^n { 2}^k\Big)
 \\
 &\ge  \exp \Big(- c_0[1 +  \delta^{3\rho+1}]2^n\Big)
\\& = \mathcal{C}_3^{2^n} ,\end{aligned} $$  
for some positive constant $\mathcal{C}_3$, which is independent of $n,p$, and $t$. 

Putting the above bounds into \eqref{bd-tn11}, we have obtained 
\begin{equation}\label{G-t123}
f(\tau+ t,p)
\geq \frac12 \mathcal{C}_1(\mathcal{C}_0 \mathcal{C}_2\mathcal{C}_3)^{2^n}  =\theta_1 e^{- \theta_2 2^n} \ge \theta_1 e^{-\theta_3 |p|^2}
\end{equation}
for all $t\in [\tau,t_*]$ and all $p$ satisfying \eqref{psqrt2}, with $\theta_1 =  \mathcal{C}_1$, $\theta_2 = \log \frac{1}{\mathcal{C}_0 \mathcal{C}_2\mathcal{C}_3}$ and $\theta_3 = \theta_2/(2\delta^2)$. Here, we stress that the constants $\theta_j$ are independent of $p$ and $t$.


~\\
{
{\bf Case 2: $|p|\le \sqrt2 \delta$.} In this case, we shall use the differential inequalities \eqref{Low-ineq} and \eqref{LossBound2}
$$
\partial_t f  \ge Q_\mathrm{gain}[f](p)  - [C_1 \mathcal{P}|p|^{\rho} + C_1\mathcal{N}|p|^{2\rho} + C_1 |p|^{3\rho+1}] f. 
$$
Therefore
\begin{eqnarray}\label{TheoremLowerBound:10}
f(t,p)&\geq &e^{-C_1 \mathcal{P}|p|^{\rho}t - C_1|p|^{2\rho}t\mathcal{N} - C_1 |p|^{3\rho+1}t}f_0(p).\end{eqnarray}
Equation \eqref{TheoremLowerBound:10} implies that for a fixed time $t_0$ and for a fixed $p_0$, then for all $|p|<|p_0|$, 
\begin{eqnarray}\label{TheoremLowerBound:11}
f(t,p)&\geq &e^{-C_1 \mathcal{P}|p_0|^{\rho}t_0 - C_1|p_0|^{2\rho}{t_0}\mathcal{N}  - C_1 |p_0|^{3\rho+1}{t_0}}f_0(p).\end{eqnarray}
Therefore, there exists $c'>0$: $f(t_0,p)>c'$ for all $|p|<|p_0|$. For each $p$, by repeating the same argument as in Case 1, in which $\delta$ is replace by $|p|/{\sqrt2}$, we can conclude that there exists $T_p$ and $b_{|p|}$ such that for all $t>T_p$, we have $$f(t,p')>b_{|p|}>0$$ for all $\sqrt2\delta\ge |p'|>|p|$ . 

Note that $f$ is continuous in $p$, then \eqref{ZeroMomentum} implies $f(t,0)= f_0(0)$. We prove that there exists a universal constant $r^*>0$ such that $f(t,p)$ is uniformly bounded from below for all $t_*\ge t\ge 0$ and $|p|\le r^*$:
\begin{equation}\label{LowerLinfty}
f(t,p)\ge C_{r^*}>0, \ \ \  \forall t_*\ge t\ge  0, |p|\le r^*.
\end{equation}

Combining the above yields the existence of a positive constant $C>0$, which is the lower bound of $f(t,p)$ in the ball $\{ |p|\le \sqrt2 \delta\}$.}

~\\
{\bf Iteration.} To conclude, we have obtained the Gaussian bound 
\begin{equation}\label{G-t1234}
f( t,p) \ge \theta_3 e^{-\theta_4 |p|^2}, \qquad p\in \RR^3, \qquad t\in [\tau,\tau+t_*], 
\end{equation}
for some universal constants $\theta_3, \theta_4$ that are independent of $p$ and $t$. Here, $t_* = 1/4$. By induction, for each integer $k\ge 1$, we then repeat the above proof, starting with initial data at $t = k t_*$. This yields the same Gaussian bound on the each time interval $[\tau+k t_*, \tau+ (k+1)t_*]$, upon noting that such a bound depends only on the  mass and second order energy-moment at $t = kt_*$, which is independent of $k^{th}$ iteration. This proves the Gaussian lower bound for all time $t\ge \tau$, and hence the main theorem.

~\\
~~ \\{\bf Acknowledgements.} 
The authors thank the referees for their constructive comments. TN's research was supported in part
by the NSF under grant DMS-1405728. M.-B Tran has been supported by NSF Grant DMS-1814149, NSF Grant RNMS (Ki-Net) 1107444. M.-B Tran would like to thank Professor Daniel Heinzen, Professor Linda Reichl, Professor Mark Raizen and Professor Robert Dorfman for fruitful discussions on the topic. The research was carried on while M.-B. Tran was visiting  University of Texas at Austin and  Penn State University. He would like to thank the institutions for the hospitality.

~~\\
{\bf Conflict of Interest:} The authors declare that they have no conflict of interest. 

  \bibliographystyle{plain}

\end{document}